\documentclass[preprint,3p]{elsarticle}
\usepackage[utf8]{inputenc}
\usepackage{amssymb}
\usepackage{amsmath}
\usepackage{amsthm}
\usepackage{tikz}

\graphicspath{{counterfigs/}}

\usetikzlibrary{calc}
\usepackage{graphicx}
\usepackage[colorinlistoftodos]{todonotes}
\usepackage[colorlinks=true, allcolors=blue]{hyperref}
\usetikzlibrary{arrows}

\usepackage{xcolor} 
\usepackage{mathdots}
\usepackage{natbib}
\usepackage[normalem]{ulem}

\usepackage{subcaption}
\usepackage{multicol}
\usepackage[justification = centering]{caption}
\usepackage{float}
\usepackage[capitalise]{cleveref}

\newtheorem{theorem}{Theorem}[section]
\newtheorem{corollary}[theorem]{Corollary}
\newtheorem{lemma}[theorem]{Lemma}

\newtheorem{proposition}[theorem]{Proposition}
\newtheorem{conjecture}[theorem]{Conjecture}

\newtheorem{question}[theorem]{Question}

\newcommand{\hatPi}{\widehat{\Pi}}
\newcommand{\hatpi}{\widehat{\pi}}
\newcommand{\Unst}{U}

\begin{document}
 
\begin{frontmatter}

\title{2-cell embeddings of cubic graphs I. The unstable dual}

\author[1]{MacKenzie Carr\corref{cor1}}
\ead{mackenzie.carr@torontomu.ca}

\author[2,3]{Bojan Mohar\thanks{On leave from:
FMF, Department of Mathematics, University of Ljubljana, Ljubljana, Slovenia.}
}
\ead{mohar@sfu.ca}

\cortext[cor1]{Corresponding author}

\affiliation[1]{organization = {Department of Mathematics, Toronto Metropolitan University}, 
city={Toronto}, 
postcode={M5B 2K3}, 
state={ON}, 
country={Canada}}

\affiliation[2]{organization = {Department of Mathematics, Simon Fraser University}, 
city={Burnaby}, 
postcode={V5A 1S6}, 
state={BC}, 
country={Canada}}

\affiliation[3]{organization = {FMF, Department of Mathematics, University of Ljubljana}, 
city={Ljubljana}, 
postcode={1000}, 
country={Slovenia}}

\fntext[fund2]{Supported in part by the NSERC Discovery Grant R832714 (Canada),
by the ERC Synergy grant (European Union, ERC, KARST, project number 101071836),
and by the Research Core Grant P1-0297 of ARIS (Slovenia).}

\begin{abstract}
     In this paper, the first of a two-part series, we explore 2-cell embeddings of cubic graphs, particularly those with small genus. Using local rotations, we introduce a new way of describing the space of 2-cell embeddings and their mutual relationship for any fixed (cubic) graph. We introduce the unstable dual of an embedding of a cubic graph, a subgraph of the dual graph, and describe how the genus of the corresponding embedding can be recovered from properties of the unstable dual. Finally, we characterize the unstable duals of embeddings with genus at most 2 of cubic cyclically 5-edge connected planar graphs and use these to generate those of genus 3 that have connectivity at most~2. 
\end{abstract}


\end{frontmatter}

\section{Introduction}

Let $G$ be a connected graph. An \emph{embedding} of $G$ in the orientable surface ${\mathbb S}_k$ of genus $k$ is a function $f:G \to {\mathbb S}_k$ that is continuous\footnote{Here, $G$ is viewed as the topological space of the corresponding 1-dimensional cell complex.} and one-to-one, so that the embedded edges intersect only at their shared endpoints. A \emph{2-cell embedding} of $G$ is an embedding in which every face is homeomorphic to an open disk. From now on, all embeddings are assumed to be 2-cell embeddings. Two embeddings $f$ and $f'$ of a graph are \emph{equivalent} if there is an orientation-preserving homeomorphism $h:{\mathbb S}_k \to {\mathbb S}_k$ such that $f'=hf$. 

Given a graph $G$, the number of (equivalence classes of) 2-cell embeddings of $G$ on the orientable surface of genus $k$ is denoted by $g_k(G)$. The sequence $\{g_k(G)\}_{k\geq 0}$ is the \emph{genus distribution} of $G$. 

Recent interest in the genus distribution of a graph stems from a proposal in \cite{GrossFurst} to develop a systematic theory of genus distributions, as well as a 1989 conjecture of Gross, Robbins and Tucker \cite{GrossRobbinsTucker} known as the Log-Concavity Genus Distribution (LCGD) Conjecture. 

\begin{conjecture}
    [\cite{GrossRobbinsTucker}] The genus distribution of every graph is log-concave.
    \label{conj:LCGD}
\end{conjecture}

We say that a sequence $\{a_k\}_{k\geq 0}$ is \emph{log-concave} if $a_i^{2}\geq a_{i-1}a_{i+1}$ for all $i\geq 1$. Note that a sequence that is log-concave must also be \emph{unimodal}, meaning that there is some integer $k$ such that $a_i\leq a_{i+1}$ for all $i<k$ and $a_i\geq a_{i+1}$ for all $i\geq k$. For this reason, a sequence that is log-concave is also referred to as a \emph{strongly unimodal} sequence. However, it is not the case that every unimodal sequence is log-concave. 

The LCGD Conjecture remained open until 2024, when Mohar \cite{MoharCounterexample} constructed graphs whose genus distributions are not log-concave. 

\begin{theorem}
[\cite{MoharCounterexample}] For every $k\geq 1$, there exists a 4-connected graph whose genus distribution has at least $k$ terms at which the log-concavity is violated. 
\label{thm:counterexample}
\end{theorem}

In the following sections, we investigate the genus distributions of cubic graphs, with the goal of using the rotation system of an embedding to determine its genus. In Section \ref{section:preliminaries}, we provide a brief overview of the relevant terminology and background. 
In Section \ref{section:framework}, we introduce a new method of describing the rotation system of an embedding, by comparing the local rotation at each vertex to the rotation system of some fixed embedding of the graph. 
In Section \ref{section:cuts and duals}, we introduce a method of describing an embedding of a cubic planar graph using the dual graph and methods for calculating the genus of an embedding using edge-cuts in the graph or particular subgraphs of the dual graph. 
In Section \ref{section:g2}, we characterize the embeddings of genus 2 in a cubic cyclically 5-edge-connected planar graph in terms of subgraphs of its dual graph. In part II of this study \cite{Part2}, we give applications of the results in part I towards a deeper understanding of why and when the log-concavity condition may hold for the genus distribution of a cubic graph. 

\section{Preliminaries}
\label{section:preliminaries}

We denote the number of vertices of a graph $G$ by $n$, i.e.~$n=|V(G)|$, and the \emph{degree} of a vertex $v\in V(G)$ by $d_G(v)$, or $d(v)$ when $G$ is clear from context. A \emph{walk} in a graph is a sequence $v_0,e_1,v_1,\dots,e_k,v_k$ of vertices and edges such that, for $1\leq i \leq k$, $e_i=v_{i-1}v_i$. We may denote a walk by only the edges in the sequence or, in a simple graph, by only the vertices in the sequence. A $(u,v)$-walk has first vertex $u$ and last vertex $v$. A \emph{path} is a walk with no repeated vertex. We denote the path graph with $n$ vertices by $P_n$. A \emph{cycle} is a walk with no repeated edge in which only the first and last vertices are the same. We denote the cycle graph with $n$ vertices by $C_n$. 

A \emph{vertex-cut} of a graph $G$ is a set $X\subseteq V(G)$ such that $G-X$ is disconnected. Then, $G$ is \emph{$k$-connected} if it contains at least $k+1$ vertices and does not contain a vertex-cut of size less than $k$. For a non-empty proper subset $X\subset V(G)$, the set of edges having one endpoint in $X$ and the other in $V(G)\setminus X$, denoted by $E(X,V(G)\setminus X)$, is an \emph{edge-cut} (or, for short, just a \emph{cut}) in $G$. An edge-cut with $k$ edges is called a \emph{$k$-cut}. Furthermore, $G$ is \emph{$k$-edge-connected} if it does not contain an edge-cut of size less than $k$. 
An edge-cut \emph{separates two cycles} in a graph if its removal results in a disconnected graph in which two distinct components contain a cycle. Accordingly, $G$ is \emph{cyclically $k$-edge connected} (C$k$EC) if no edge-cut of size less than $k$ separates two cycles in $G$. A cycle $C$ is a \emph{separating cycle} in $G$ if $G-V(C)$ is disconnected. 

Given a set of vertices $S\subseteq V(G)$, the \emph{subgraph induced by $S$}, denoted by $G[S]$, is the graph with vertex set $S$ and containing every edge in $E(G)$ whose endpoints are both in $S$. An \emph{induced cycle} is an induced subgraph that is a cycle. Given an embedding $\Pi$ of a graph $G$, if $H$ is a connected subgraph of $G$, then the \emph{induced embedding} of $H$ is obtained from $\Pi$ by ignoring all of the edges in $E(G)\setminus E(H)$. 

The \emph{(minimum) genus} of a graph $G$ is the minimum $k$ for which $G$ has a 2-cell embedding in the orientable surface ${\mathbb S}_k$. A \emph{planar} graph is a graph with minimum genus 0 and a \emph{toroidal} graph has minimum genus 1. Similarly, the \emph{maximum genus} of $G$ is the maximum $M$ for which $G$ has a 2-cell embedding on the orientable surface of genus $M$. Given an embedding $\Pi$ of a graph $G$, we denote the genus of $\Pi$ by $g(\Pi)$. 

A clockwise \emph{local rotation} at a vertex $v\in V(G)$ is a (clockwise) cyclic ordering of the edges incident with $v$. In an embedding $\Pi$, we denote the local rotation at the vertex $v$ by $\pi_v$. A \emph{rotation system} is a set of rotations, one for each vertex of the graph. Two embeddings are equivalent if and only if their corresponding rotation systems are the same. From this, we conclude that a graph $G$ has a total of $\prod_{v\in V(G)} (d(v)-1)!$ distinct embeddings. 

Let $\Pi$ be an embedding of a graph $G$ in an orientable surface $S$. A \emph{facial walk} is a walk in $G$ that bounds a face of the embedding. Similarly, a \emph{facial cycle} is a cycle in $G$ that bounds a face of the embedding. Two embeddings of a graph are equivalent if and only if their clockwise-directed facial walks are the same. A cycle $C$ of $G$ is a \emph{surface-separating} cycle if it splits the surface into two components. This is the same as being \emph{zero-homologous} (the trivial element in the $\mathbb{Z}_2$-homology of the 2-dimensional cell complex of the surface \cite{hatcher}). As we will be dealing with more general elements of the cycle space of the graph, we will use the latter terminology. A \emph{contractible} cycle $C$ of $G$ separates the surface into two components, one of which is a disk. Equivalently, $C$ is a contractible cycle if it is contractible as a curve on the surface. 

An embedding $\Pi$ of a graph $G$ is a \emph{polyhedral} embedding if every $\Pi$-facial walk is a facial cycle and the intersection of every pair of $\Pi$-facial walks is either empty, a single vertex, or a single edge. 
Two vertices or two edges are \emph{cofacial} if there exists some facial walk containing both. 

A fundamental result of Tutte shows the relationship between the connectivity of a graph and the structure of the facial cycles in a planar embedding. 

\begin{theorem}
[\cite{Tutte1963}]
Let $G$ be a 3-connected planar graph. A cycle $C$ of $G$ is a facial cycle of some planar embedding of $G$ if and only if it is induced and non-separating. Moreover, $C$ is a facial cycle in every planar representation of $G$. 
\label{thm:Tutte}
\end{theorem}

The induced non-separating cycles of a graph are called \emph{peripheral cycles}. Thus, Theorem~\ref{thm:Tutte} says that the peripheral cycles of a 3-connected planar graph $G$ are exactly the facial cycles of a planar embedding of $G$. 

Cycles that are graph non-separating may be surface separating cycles in a particular embedding. However, this cannot happen for peripheral cycles. 

\begin{lemma}
    [\cite{MoharCounterexample}] Let $C$ be a peripheral cycle of $G$ that is embedded in a surface $S$. If $C$ is not facial, then it is not zero-homologous. 
    \label{lem:surfaceNonSep}
\end{lemma}

With the additional condition that the graph $G$ is C5EC, pairs of peripheral cycles are together non-separating in $G$. 

\begin{lemma}
Let $G$ be a cubic C5EC planar graph with peripheral cycles $C_1$ and $C_2$. Then, $G-V(C_1\cup C_2)$ is connected. 
\label{lem:G-c1c2 conn}
\end{lemma}

\begin{proof}
Suppose not, i.e.~that $G-V(C_1\cup C_2)$ is disconnected. Let $A$ be some component of $G-V(C_1\cup C_2)$. Since $C_1$ and $C_2$ are peripheral cycles in $G$, they are each non-separating and thus there is at least one edge from $A$ to $C_1$ and at least one edge from $A$ to $C_2$. 

We claim that the vertices of $C_1$ that have a neighbor in $A$ are consecutive vertices on the cycle $C_1$. 
Suppose, for a contradiction, that there are vertices $a_1$ and $a_2$ on $C_1$ that have neighbors in $A$ and vertices $b_1$ and $b_2$ on $C_1$ whose neighbors are not in $A$, with $C_1 = a_1\dots b_1\dots a_2\dots b_2\dots$. 

Consider a planar embedding $\Pi_0$ of $G$. By Theorem~\ref{thm:Tutte}, the facial cycles in $\Pi_0$ are exactly the peripheral cycles of $G$. Let $F$ be the face whose boundary is the cycle $C_1$. 

There are paths on $C_1$ from $a_1$ to $b_1$, $b_1$ to $a_2$, $a_2$ to $b_2$, and $b_2$ to  $a_1$. Since $A$ is a connected component of $G-V(C_1\cup C_2)$, there is a path from $a_1$ to $a_2$ whose interior vertices are all in $A$. Finally, there is a path from $b_1$ to $C_2$ and from $b_2$ to $C_2$ in $G$, none of whose interior vertices are in $A$. Thus, there is a path from $b_1$ to $b_2$ in $G$. 
This contradicts the planarity of $G$ and completes the proof of the claim. 

Now, the vertices of $C_1$ that are adjacent to vertices in $A$ must be consecutive on $C_1$, say $x_1, x_2, \dots, x_i$. By a similar argument, the vertices of $C_2$ that are adjacent to vertices in $A$ must be consecutive on $C_2$, say $y_1, y_2, \dots, y_j$. We will call these sequences of vertices the $A$\emph{-interval} on $C_1$ and $C_2$, respectively. 

Let $x_0$ and $x_{i+1}$ be the neighbors of $x_1$ and $x_i$ on $C_1$ that are not in the $A$-interval on $C_1$. Similarly, let $y_0$ and $y_{j+1}$ be the neighbors of $y_1$ and $y_j$ on $C_2$ that are not in the $A$-interval on $C_2$. We show now that the edge set $S = \{x_0x_1, x_ix_{i+1},y_0y_1, y_jy_{j+1}\}$ separates two cycles in $G$. 

Let $B$ be any other component of $G-V(C_1\cup C_2)$. Note that any path from a vertex of $A$ to a vertex of $B$ must contain a vertex of the $A$-interval on $C_1$ or $C_2$. Since the graph $G$ is cubic, the $A$-interval and $B$-interval are disjoint. Thus, the set $S$ separates $N_G[V(A)]$ and $N_G[V(B)]$. We denote these subgraphs by $N_A$ and $N_B$, respectively. Note that the $A$-intervals on $C_1$ and $C_2$ are both contained in $N_A$. 

If $A$ (respectively, $B$) contains a cycle, then clearly $N_A$ (resp.~$N_B$) does as well. 
If $A$ does not contain a cycle, then consider some vertex $a\in V(A)$ that has at most one neighbor in $A$. Since $V(A)$ induces a tree, such a vertex must exist. If $a$ has two neighbors $u$ and $v$ in $C_1$, then let $uPv$ be the path on $C_1$ from $u$ to $v$ whose vertices are all contained in the $A$-interval on $C_1$. The path $uPv$ together with the edges $au$ and $av$ form a cycle in $N_A$. If $a$ has exactly one neighbor, say $u$, in $C_1$ and exactly one neighbor in $C_2$, then there must exist some other vertex, say $a^*$, with at most one neighbor in $A$. If $a^*$ has two neighbors on one of $C_1$ or $C_2$, then we can construct a cycle in $N_A$ as above. If $a^*$ has exactly one neighbor, say $w$ on $C_1$ and exactly one neighbor on $C_2$, then we construct a cycle in $N_A$ as follows. Let $aQa^*$ be the path in $A$ from $a$ to $a^*$ and let $uPw$ be the path on $C_1$ from $u$ to $w$ whose vertices are all contained in the $A$-interval on $C_1$. Then, the paths $aQa^*$ and $uPw$ together with the edges $au$ and $a^*w$ form a cycle in $N_A$. 

A similar argument shows that there exists a cycle in $N_B$. Since $S$ separates $N_A$ and $N_B$, it is a set of at most four edges separating two cycles in $G$. This contradicts $G$ being C5EC. Thus, $G-V(C_1\cup C_2)$ is connected. 
\end{proof}

A family of cycles $\mathcal{C}$ is said to be a \emph{peripheral family} of cycles if its members are pairwise disjoint, nonadjacent, and $G-\bigcup\{V(C)\mid C\in \mathcal{C}\}$ is connected. In a particular embedding $\Pi$ of genus $g(\Pi)$, there is a limit on the maximum size of a peripheral family of $\Pi$-nonfacial cycles in $G$. 

\begin{lemma}
    [\cite{MoharCounterexample}] Let $\mathcal{C}$ be a peripheral family of cycles of $G$ and $\Pi$ an embedding of $G$. If none of the cycles in $\mathcal{C}$ is $\Pi$-facial, then $|\mathcal{C}|\leq g(\Pi)$. 
    \label{lem:numCycles}
\end{lemma}

Similarly, there is an upper bound on the number of disjoint nonhomotopic cycles on a particular surface. Two cycles of a graph embedded in a surface are \emph{homotopic} if they are freely homotopic as closed curves on the surface. 

\begin{proposition}
    [\cite{MoharThomassen}] \label{prop:disjointCycles} Let $G$ be a graph and $\Pi$ an embedding of $G$. Let $C_1,C_2,\dots,C_k$ be pairwise disjoint, $\Pi$-noncontractible, pairwise $\Pi$-nonhomotopic cycles of $G$. Then $$
k\leq\begin{cases}
      g(\Pi) & g(\Pi)\leq 1 \\
      3g(\Pi)-3 & g(\Pi)\geq 2~.
    \end{cases}
$$
\end{proposition}

A graph being cubic, C5EC and planar guarantees that in a toroidal embedding $\Pi$, the disjoint $\Pi$-nonfacial peripheral cycles cannot be too far apart. In fact, they must be joined by a single edge. 

\begin{lemma}
Let $G$ be a cubic C5EC planar graph and $\Pi$ be a toroidal embedding of $G$. If two disjoint peripheral cycles, $C_1$ and $C_2$ of $G$ are $\Pi$-nonfacial, then there is exactly one edge joining a vertex of $C_1$ to a vertex of $C_2$ in $G$. 
\label{lem:one edge}
\end{lemma}

\begin{proof}
Let $C_1$ and $C_2$ be disjoint peripheral cycles in $G$ that are both $\Pi$-nonfacial. By Lemma~\ref{lem:G-c1c2 conn}, $G-V(C_1\cup C_2)$ is connected and therefore, by Lemma~\ref{lem:numCycles}, $C_1$ and $C_2$ must be adjacent. Suppose that there are two distinct edges, $uv$ and $wz$ in $G$ with $u,w\in V(C_1)$ and $v,z\in V(C_2)$. Note that $u$ and $w$ (similarly, $v$ and $z$) must be distinct vertices because $G$ is cubic and each vertex has two neighbors in $C_1$ (resp.~$C_2$). We argue that we can choose these two edges so that $u$ and $w$ are adjacent on $C_1$ and $v$ and $z$ are adjacent on $C_2$. 

Let $A = G-V(C_1\cup C_2)$. Suppose that there are vertices $a_1$ and $a_2$ with neighbors in $A$ and vertices $b_1$ and $b_2$ with neighbors on $C_2$ such that $C_1 = a_1\dots b_1\dots a_2\dots b_2\dots$. Then, as in the proof of Lemma~\ref{lem:G-c1c2 conn}, there is a path from $a_1$ to $a_2$ whose interior vertices are all in $A$, and a path from $b_1$ to $b_2$, none of whose interior vertices are in $A$, contradicting the planarity of $G$. 
Thus, the vertices of $C_1$ that have neighbors in $A$ are all consecutive vertices on $C_1$ and, by similar argument, the vertices of $C_2$ that have neighbors in $A$ are all consecutive vertices on $C_2$. 

In addition, this means that the vertices of $C_1$ with neighbors in $C_2$ are all consecutive on $C_1$ and, similarly, the vertices of $C_2$ with neighbors in $C_1$ are all consecutive on $C_2$. 

As shown in the proof of Lemma~\ref{lem:G-c1c2 conn}, the graph induced by $N_G[V(A)]$ contains a cycle. Now, consider the set of edges $S = \{x_0x_1,x_ix_{i+1},y_0y_1,y_jy_{j+1}\}$, where $x_1,x_i\in V(C_1)$ have neighbors in $A$, $x_0,x_{i+1}\in V(C_1)$ have neighbors in $C_2$ and, similarly, $y_1y_j\in V(C_2)$ have neighbors in $A$ and $y_0,y_{j+1}\in V(C_2)$ have neighbors in $C_1$. Then, the set $S$ separates $N_G[V(A)]$ from the subgraph $H$ induced by the endpoints of edges connecting a vertex in $C_1$ with a vertex in $C_2$. Now, if there are at least two such edges $uv$ and $wz$ in $G$ with $u,w\in V(C_1)$ and $v,z\in V(C_2)$, then $H$ contains a cycle. Thus, $S$ is a set of at most four edges that separates two cycles in $G$, contradicting $G$ being C5EC. Therefore, there must be exactly one edge joining a vertex of $C_1$ and a vertex of $C_2$, as claimed. 
\end{proof}

\section{A framework for embeddings of cubic graphs}
\label{section:framework}

In this section, we provide a new method for describing and analyzing the embeddings of a cubic graph using the set of possible rotation systems. Given a cubic graph $G$, fix an embedding $\hatPi$ of $G$ and call this a \emph{base embedding} of $G$. Let $\hatpi_v$ be the local rotation in $\hatPi$ for each vertex $v\in V(G)$. Now, consider another embedding $\Pi$ of $G$. For each vertex $v\in V(G)$, we say that $v$ is a \emph{$\Pi$-unstable vertex} with respect to $\hatPi$ if the local rotation $\pi_v \neq \hatpi_v$. If a vertex is not $\Pi$-unstable, then it is \emph{$\Pi$-stable}. We define the \emph{$\Pi$-unstable set} with respect to $\hatPi$ to be the set $\Unst_\Pi = \{v\in V(G) \mid v \text{~is~} \Pi\text{-unstable with respect to }\hatPi\}$. See Figure \ref{fig:unstableEx1} for an example. 

\begin{figure}[!h]
    \centering
    \begin{subfigure}[b]{0.35\textwidth}
         \centering
         \includegraphics[width=0.67\textwidth]{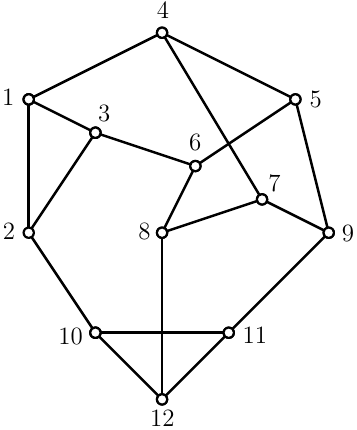}
         \caption{The base embedding $\hatPi$ of $G$}
         \label{fig:unstableExbase}
     \end{subfigure}
     \hspace{1cm}
     \begin{subfigure}[b]{0.35\textwidth}
         \centering
         \includegraphics[width=0.7\textwidth]{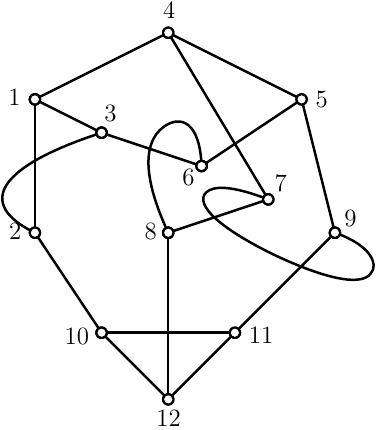}
         \caption{The embedding $\Pi$ of $G$}
         \label{fig:unstableEx2}
     \end{subfigure}
    \caption{A base embedding of $G$ and an embedding $\Pi$ with $\Pi$-unstable set $\Unst_\Pi = \{2, 6, 7, 9\}$. The local rotation at each vertex is determined by the cyclic clockwise order as shown in the figure.}
    \label{fig:unstableEx1}
\end{figure}

Now, define a \emph{$\Pi$-unstable edge} (with respect to $\hatPi$) to be an edge with exactly one $\Pi$-unstable endpoint with respect to $\hatPi$. In other words, it is an edge $e=uv\in E(G)$ connecting a $\Pi$-stable vertex $u$ and a $\Pi$-unstable vertex $v$. 
Using this concept of stable and unstable edges, we describe a method of tracing the faces of the embedding $\Pi$. We add a signature $f:E(G)\to\{-1,1\}$ to each of the edges of the graph. For an embedding $\Pi$ and an edge $e\in E(G)$, let $f(e)=1$ if $e$ is a $\Pi$-stable edge, and $f(e)=-1$ if $e$ is $\Pi$-unstable. When tracing a $\Pi$-facial walk, we use a variable $t_i$ to indicate the local rotation used at the next vertex in the facial walk. Initially, we begin at some vertex in the facial walk and set $t_0=1$. Then, following each edge $e$ from $u$ to $v$ in the facial walk, we set $t_{i+1}=t_if(e)$. The vertex $v$ is followed in the facial walk by $\hatpi_v(e)$ if $t_{i+1}=1$ and by ${\hatpi_v}^{-1}(e)$ if $t_{i+1}=-1$. Informally, we can think of the facial walk crossing to the other side of a $\Pi$-unstable edge in the embedding $\hatPi$ each time such an edge is encountered. If we think of a $\Pi$-facial walk ``leaving'' the face by crossing one of the unstable edges, and ``entering'' the face using another, it is clear that any face of the base embedding whose boundary contains a $\Pi$-unstable edge must contain an even number of $\Pi$-unstable edges.

We note the similarity between this method of tracing a $\Pi$-facial walk and a Petrie walk \cite{petrie}, also called a left-right walk, for example in \cite{leftrightwalks}, a walk whose edges alternate turning to the left edge and to the right edge of the current edge in the local rotation. The face tracing method described above is a variation of a Petrie walk in which the direction taken changes only after traversing a $\Pi$-unstable edge, rather than after traversing every edge.

Consider a face $F$ in the base embedding $\hatPi$ of a graph $G$. If $F$ is also a face in an embedding $\Pi$ of $G$, then $F$ is a \emph{$\Pi$-stable} face with respect to $\hatPi$, regardless of whether the face is traversed in the same direction in both $\Pi$ and $\hatPi$. Otherwise, $F$ is a \emph{$\Pi$-unstable face}. Note that the $\Pi$-unstable faces are those that contain at least one $\Pi$-unstable edge.

In an embedding $\Pi$ of a cubic graph, there are four possible ways that the faces of the embedding can appear at a given vertex $v\in V(G)$. If $v$ appears three times on a single $\Pi$-face, we say that $v$ is of Type 1 in $\Pi$. Each of the edges incident with $v$ appears twice on this facial walk and there are two different ways this can occur: the two appearances of each edge in the facial walk can be consecutive ($e_1,v,e_2,\dots, e_2,v,e_3,\dots, e_3,v,e_1,\dots$) or nonconsecutive ($e_1,v,e_2,\dots, e_3,v,e_1,\dots, e_2,v,e_3,\dots$). We will call these two types Type 1A and Type 1B, respectively. If $v$ appears twice on a single $\Pi$-face and once on another, we say that $v$ is of Type 2 in $\Pi$. And if $v$ appears on three distinct $\Pi$-faces, we say that $v$ is of Type 3 in $\Pi$. 

Now, in each of these cases, we can describe the changes to the facial structure at $v$ when the local rotation at $v$ is changed, but the local rotation at every other vertex stays the same. These effects are shown in Figure \ref{fig:cubicv}: a vertex of Type 3 becomes a vertex of Type 1B (Figure \ref{fig:type3}), a vertex of Type 1B becomes a vertex of Type 3 (Figure \ref{fig:type1B}), and vertices of Type 1A and Type 2 remain vertices of the same type (Figures \ref{fig:type1A} and \ref{fig:type2}, respectively). 

\begin{figure}[ht]
    \centering
    \begin{subfigure}[b]{0.4\textwidth}
         \centering
         \includegraphics[width=\textwidth]{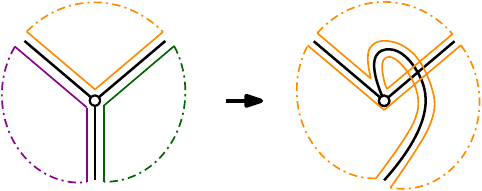}
         \caption{Type 3 becomes Type 1B}
         \label{fig:type3}
     \end{subfigure}
     \hspace{1cm}
     \begin{subfigure}[b]{0.4\textwidth}
         \centering
         \includegraphics[width=\textwidth]{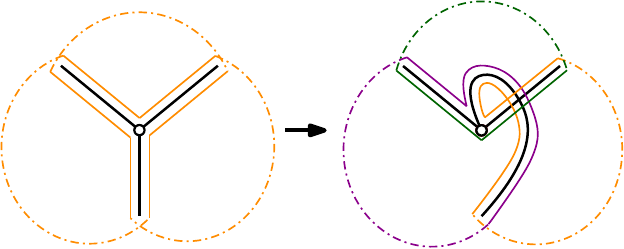}
         \caption{Type 1B becomes Type 3}
         \label{fig:type1B}
     \end{subfigure}
     \vspace{7mm}

     \begin{subfigure}[b]{0.38\textwidth}
         \centering
         \includegraphics[width=\textwidth]{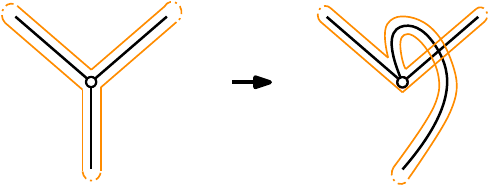}
         \caption{Type 1A remains Type 1A}
         \label{fig:type1A}
     \end{subfigure}
     \hspace{1cm}
     \begin{subfigure}[b]{0.4\textwidth}
         \centering
         \includegraphics[width=\textwidth]{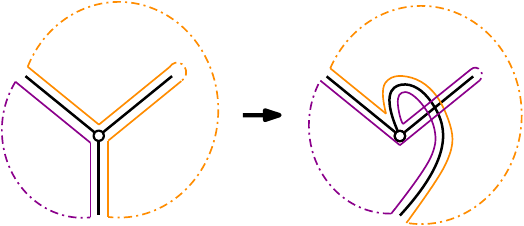}
         \caption{Type 2 remains Type 2}
         \label{fig:type2}
     \end{subfigure}
    \caption{Each of the four possible face structures at a vertex $v$ in an embedding of a cubic graph and the effect of changing the local rotation only at $v$ to obtain a new embedding}
    \label{fig:cubicv}
\end{figure}

By changing the local rotation only at the vertex $v$, only the faces containing $v$ are affected. Similar to the method of recombinant strands in \cite{Gross2010}, the edges incident with $v$ break the face(s) containing $v$ into strands, or subwalks, of the facial walk(s) in the original embedding, say $\Pi_1$. After modifying the local rotation at $v$, these subwalks get reconnected according to the new local rotation, forming facial walks in the new embedding $\Pi_2$. If $v$ is of Type 3 in $\Pi_1$, then the total number of faces in $\Pi_2$ is two fewer than in $\Pi_1$. Thus, the genus of the new embedding $\Pi_2$ is one more than that of $\Pi_1$. Similarly, if $v$ is of Type 1B, the genus of $\Pi_2$ is one less than that of $\Pi_1$. Changing the local rotation at a vertex of Type 1A or Type 2 does not change the genus because the number of faces is unchanged, though the structure of these faces may be different. 

Returning to the notion of stability with respect to a given base embedding $\hatPi$, we now have a way of determining the genus of an embedding $\Pi$ that has a single $\Pi$-unstable vertex. We simply determine the type of the unstable vertex in the embedding $\Pi$ with respect to $\hatPi$ and apply the conclusions above.

For a given base embedding $\hatPi$ of a cubic graph $G$, every embedding can be described by its unstable vertices with respect to $\hatPi$. This naturally suggests a lattice or hypercube structure that we describe as follows. Label the vertices of $G$ by $V(G) = \{v_1,v_2,\dots,v_n\}$. Then, for an embedding $\Pi$ of $G$, we can encode the embedding with an $n$-tuple, where the $i^{th}$ entry is 0 if the vertex $v_i$ is $\Pi$-stable, and 1 if $v_i$ is $\Pi$-unstable. The base embedding $\hatPi$ corresponds to the $n$-tuple $(0,0,\dots,0)$. The $n$-tuple $(1,1,\dots,1)$ corresponds to the embedding in which every vertex is $\Pi$-unstable, but every edge and every face is $\Pi$-stable, with the facial walks of $\hatPi$ oriented in the opposite direction. 

We define the \emph{configuration graph of embeddings} with respect to $\hatPi$ to be the graph whose vertices are the binary $n$-tuples, with two $n$-tuples adjacent if and only if they differ in a single position. 
Clearly, the configuration graph is isomorphic to the $n$-dimensional hypercube $Q_n$. We note the similarity of the configuration graph to the \emph{stratified graph} of a graph $G$, introduced by Gross and Tucker \cite{stratified}. 

\begin{figure}[ht]
    \centering
            \includegraphics[width=0.78\textwidth]{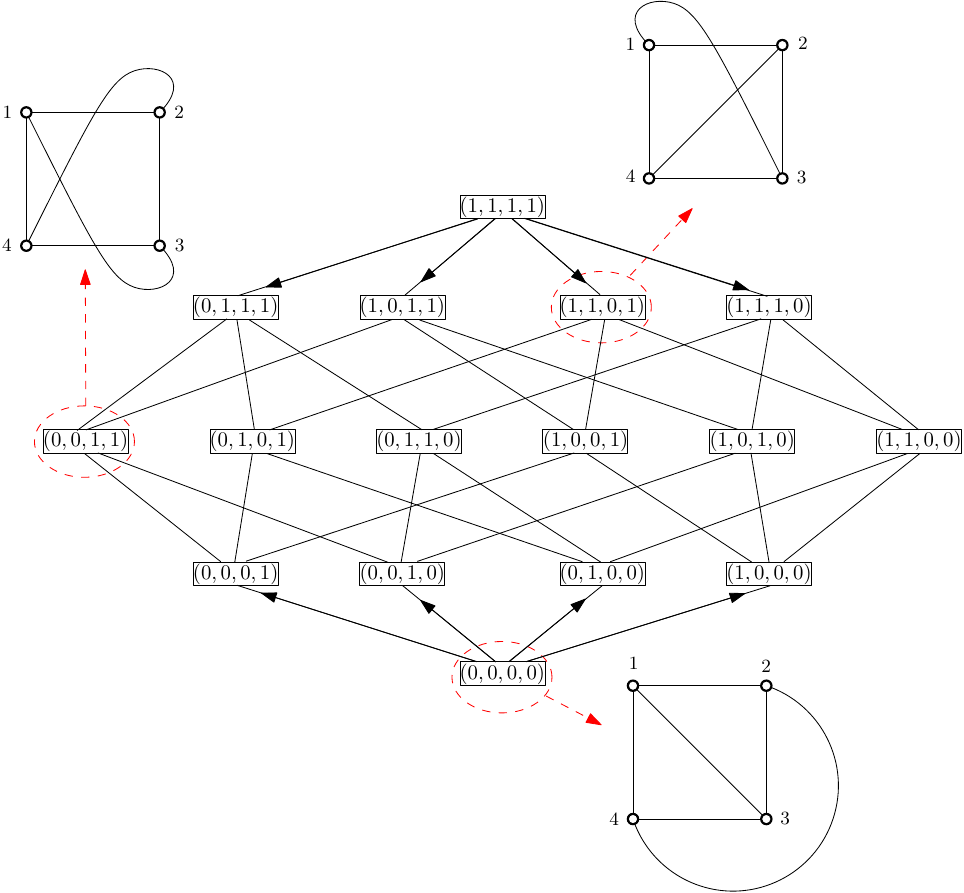}
            \caption{The configuration graph of $K_4$ with respect to the base embedding with local rotations $\hatpi_1 = (2,3,4)$, $\hatpi_2 = (4,3,1)$, $\hatpi_3 = (2,4,1)$ and $\hatpi_4 = (1,3,2)$. Two other embeddings are shown, corresponding to 4-tuples $(0,0,1,1)$ and $(1,1,0,1)$. The directed edges indicate where the genus increases.}
    \label{fig:stabilitygraph1}
\end{figure}

Each edge in the configuration graph corresponds to a change in local rotation at exactly one vertex. Using the method described above for determining the change in genus by changing the local rotation or stability of exactly one vertex, we direct the edges of the configuration graph to indicate the change in genus. For an edge between embeddings $\Pi_a$ and $\Pi_b$, we direct the edge $\Pi_a \to \Pi_b$ if $g(\Pi_a) + 1 = g(\Pi_b)$. The edge remains undirected if $g(\Pi_a) = g(\Pi_b)$. 
See Figure~\ref{fig:stabilitygraph1}. 

Recall that, by Theorem \ref{thm:Tutte}, a 3-connected planar graph $G$ has exactly two planar embeddings, both with the same set of faces and the facial walks of the two embeddings oriented in opposite directions. Consider one of these planar embeddings to be the base embedding $\hatPi$ of $G$. 
Every vertex in a planar embedding of $G$ is of Type 3, so every edge incident with the vertex $(0,0,\dots,0)$ in the configuration graph will be directed away from $(0,0,\dots,0)$ and every edge incident with $(1,1,\dots,1)$ will similarly be directed away from this vertex. For this reason, we will assume that the base embedding for a 3-connected planar graph is one of the two planar embeddings, unless otherwise specified. 

The structure of the unstable vertices, edges and faces in an embedding of a cubic graph can give some insight into the calculation of the genus distribution. We explore two such structures in the following section. 

\section{Unstable cuts and unstable duals}
\label{section:cuts and duals}

\subsection{Unstable cuts}

Given an embedding $\Pi$ of a cubic graph $G$, let $\Unst_\Pi$ be the $\Pi$-unstable set as defined above. Then the \emph{$\Pi$-unstable cut}, denoted by $\delta(\Unst_\Pi)$, is the edge-cut $E(\Unst_\Pi,V(G)\setminus \Unst_\Pi)$. Note that, because $\delta(\Unst_\Pi)$ contains exactly the edges with one $\Pi$-unstable endpoint and one $\Pi$-stable endpoint, it follows that $\delta(\Unst_\Pi)$ is the set of all $\Pi$-unstable edges. The subgraph induced by the $\Pi$-unstable set, $G[\Unst_\Pi]$, has components $H_1,H_2,\dots,H_\gamma$. We call these the \emph{$\Pi$-unstable components} of $G$ and, for $i=1,2,\dots,\gamma$, let $c_i = |E\left(V(H_i),V(G)\setminus V(H_i)\right)|$, i.e.~the number of $\Pi$-unstable edges whose $\Pi$-unstable endpoint is in the component $H_i$. It follows that $|\delta(\Unst_\Pi)| = \sum_{i=1}^{\gamma} c_i$. See Figure \ref{fig:UnstableEx} for an example with three $\Pi$-unstable components and $|\delta(\Unst_\Pi)|=13$. 

\begin{figure}[ht]
    \centering
         \hspace{-1cm}\includegraphics[width=0.45\textwidth]{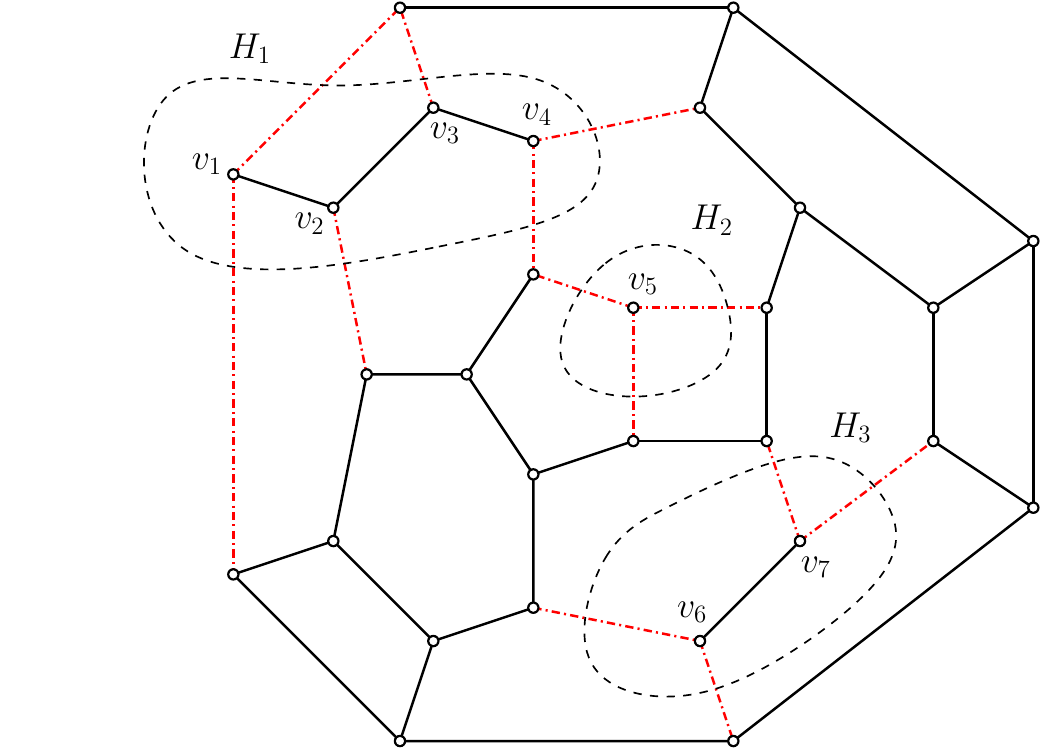}        
    \caption{A planar embedding of a cubic graph $G$ with unstable vertices in an embedding $\Pi$ labelled $v_1,v_2,\dots,v_7$. Unstable edges are indicated by red dash dotted edges. The three $\Pi$-unstable components, $H_1$, $H_2$ and $H_3$ are circled using dashed lines.}
    \label{fig:UnstableEx}
\end{figure}

Consider a cubic 3-connected planar graph $G$. If an embedding $\Pi$ of $G$ has a single $\Pi$-unstable component, then by tracing the $\Pi$-facial walks we can see that the planar faces containing the $\Pi$-unstable edges incident with that component become either one or two $\Pi$-faces, depending on the parity of the number of planar faces. It is more complex to study the $\Pi$-facial walks when there are multiple unstable components whose unstable edges appear together on a face of the base embedding, or when the base embedding is not planar. However, we can determine the maximum possible decrease in the number of faces, to obtain an upper bound on the genus $g(\Pi)$ of the embedding $\Pi$.

\begin{theorem}
    \label{thm:edgecutsupper}
    Let $G$ be a cubic graph and let $\Pi$ be an embedding of $G$ with $\Pi$-unstable cut $\delta(\Unst_\Pi)$ with respect to a base embedding $\hatPi$ of $G$. Let $H_1,H_2,\dots,H_\gamma$ be the $\Pi$-unstable components with $c_i = |E(V(H_i),V(G)\setminus V(H_i))|$, for $i=1,2,\dots,\gamma$. Then, $$g(\Pi) \leq g(\hatPi) + \left\lfloor \frac{c_1-1}{2}\right\rfloor +\left\lfloor \frac{c_2-1}{2}\right\rfloor + \dots + \left\lfloor \frac{c_\gamma-1}{2}\right\rfloor.$$
\end{theorem}

\begin{proof}
    Let $E_i = E(V(H_i),V(G)\setminus V(H_i))$, for $1\leq i \leq \gamma$. Note that these edge-sets partition the set of $\Pi$-unstable edges. We obtain the embedding $\Pi$ by beginning with $\hatPi$ and, for $i=1,2,\dots,\gamma$, flipping the local rotation at each vertex of $H_i$. It suffices to show that at each such step, the genus of the embedding increases by at most $\left\lfloor \frac{c_i-1}{2} \right\rfloor$.

    At stage $i$, only the facial walks containing the unstable edges in $E_i$ are affected; all other facial walks remain unchanged. The number of these facial walks is at most $c_i$ and hence the number of faces decreases by at most $c_i-1$. Then, by Euler's formula, the genus increases by at most $\left\lfloor \frac{c_i-1}{2} \right\rfloor$. 

    Therefore, $g(\Pi)\leq g(\hatPi) + \sum_{i=1}^{\gamma} \left\lfloor \frac{c_i-1}{2}\right\rfloor.$
\end{proof}

In some cases, additional conditions on the connectivity of $G$ or the genus of the base embedding $\hatPi$ allow for an exact calculation of the genus of an embedding with a small number of unstable edges. 

\begin{corollary}
    Let $G$ be a cubic 3-connected planar graph with $n\geq 5$. Let $\Pi$ be an embedding of $G$ with $\Pi$-unstable cut $\delta(\Unst_\Pi)$. If $|\delta(\Unst_\Pi)|\in\{3,4\}$, then $g(\Pi)=1$. 
    \label{cor:g1cut}
\end{corollary}

\begin{proof}
    Since $G$ is 3-connected, $|\delta(\Unst_\Pi)|\in\{3,4\}$ implies that there is exactly one $\Pi$-unstable component $H_1$ with $c_1 = |\delta(\Unst_\Pi)|$. By Theorem~\ref{thm:edgecutsupper}, $$g(\Pi) \leq \left\lfloor \frac{c_1-1}{2}\right\rfloor = 1.$$
    Since the only planar embeddings of $G$ have no $\Pi$-unstable edges, it must be the case that $g(\Pi)=1$. 
    \end{proof}
    
Note that in a simple cubic graph, a single unstable vertex or a pair of adjacent unstable vertices results in an unstable cut of size 3 or 4, making these candidates for embeddings with genus 1 when the base embedding is planar. 
    
There are some cases where the bound in Theorem~\ref{thm:edgecutsupper} is tight. For example, if $G$ is a cubic 3-connected planar graph and $\Pi$ is an embedding of $G$ with three unstable vertices that induce a path of length 2. In this case, there is a single $\Pi$-unstable component with five incident unstable edges: two incident with each of the endpoints of the path and one incident with the midpoint of the path. Then, by Theorem \ref{thm:edgecutsupper}, the genus of this embedding is at most $\left\lfloor \frac{5-1}{2}\right\rfloor = 2$. Tracing the faces of this embedding, we see that this is exact. See Figure \ref{fig:5cutb}.

\begin{figure}[ht]
    \centering
    \begin{subfigure}[b]{0.4\textwidth}
         \centering
         \includegraphics[width=0.9\textwidth]{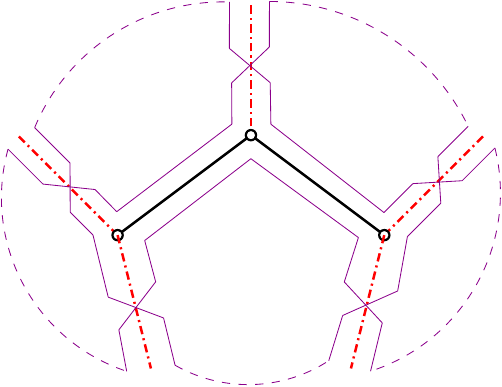}
         \caption{$\Pi$-unstable component is a path of length 2}
         \label{fig:5cutb}
     \end{subfigure}
     \hspace{1.5cm}
     \begin{subfigure}[b]{0.4\textwidth}
         \centering
         \includegraphics[width=0.9\textwidth]{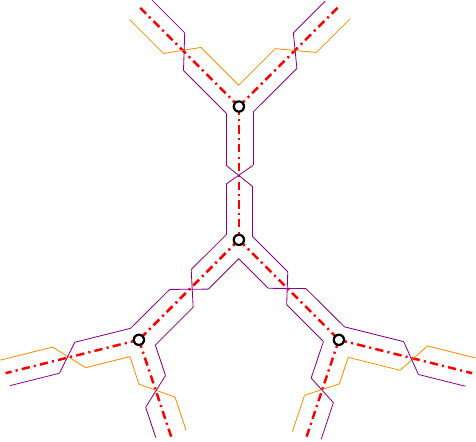}
         \caption{Three $\Pi$-unstable components with a single vertex}
         \label{fig:8cut}
     \end{subfigure}
     \caption{An example of a $\Pi$-unstable set for which Theorem \ref{thm:edgecutsupper} is tight and one for which it is not}
    \label{fig:tightex}
\end{figure}

However, consider a cubic 3-connected planar graph with a vertex $v$ whose closed neighborhood induces a claw (i.e.~there is no edge between any pair of neighbors of $v$). If $\Pi$ is an embedding of $G$ with $\Unst_\Pi = N(v)$, then we have three $\Pi$-unstable components, each with three unstable edges. By Theorem \ref{thm:edgecutsupper}, the genus of this embedding is at most $3\left\lfloor \frac{3-1}{2} \right\rfloor =3$. However, by tracing the faces, we see that the number of faces is reduced by 4, so $g(\Pi)=2$. See Figure \ref{fig:8cut}. In this case, Theorem \ref{thm:edgecutsupper} is not tight because the unstable components are ``too close'' in the graph, causing the facial walks incident with one unstable component in the planar embedding to be combined with the facial walks incident with another unstable component when forming the $\Pi$-facial walks. 

\subsection{Unstable dual}

The dual edges of an edge-cut in a 2-cell embedded graph $G$ form an Eulerian subgraph $H^*$ of the dual graph $G^*$. They separate the surface into two parts, $A\cup B$, with $\partial A = \partial B = E(H^*)$. Such Eulerian subgraphs of the dual are said to be \emph{zero-homologous} since they correspond to zero homology class in the cycle space of $G^*$. Recall that an Eulerian graph is a graph in which every vertex has even degree. Note that we allow an Eulerian graph to be disconnected. 

Given a base embedding $\hatPi$ of a cubic graph $G$ and an embedding $\Pi$ of $G$, we construct the \emph{$\Pi$-unstable dual} graph of $G$, denoted by $G^*_\Pi$ as follows. For a set of edges $U\subseteq E(G)$, let $U^*$ be the set of corresponding dual edges. Let $B\subseteq E(G)$ and $D\subseteq F(G)$ be the set of $\Pi$-unstable edges and the set of $\Pi$-unstable faces in $G$, respectively, with respect to $\hatPi$. Then $G^*_\Pi$ is the graph with vertex set $V(G^*_\Pi)=D$ and edge set $E(G^*_\Pi)=B^*$. Note that $G^*_\Pi$ is a subgraph (not necessarily induced) of the dual graph $G^*$, where the dual graph is constructed from $\hatPi$. 

\begin{figure}[ht]
    \centering
         \includegraphics[width=0.45\textwidth]{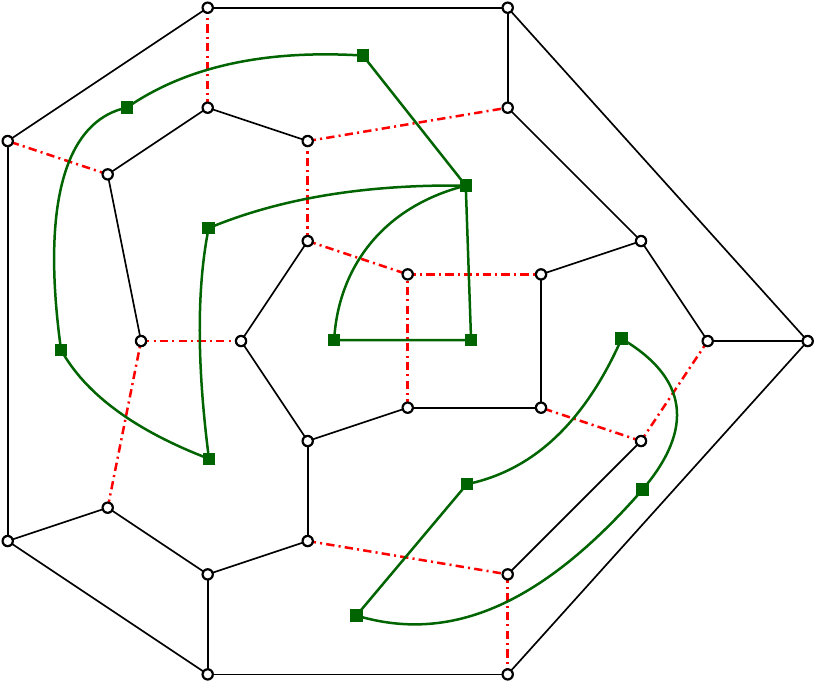}        
    \caption{A cubic 3-connected planar graph $G$ with $\Pi$-unstable edges for some embedding $\Pi$ indicated by red dash dotted edges. The $\Pi$-unstable dual is drawn in green, with square vertices and solid edges. }
    \label{fig:unstabledualexample1}
\end{figure}

Consider the example shown in Figure~\ref{fig:unstabledualexample1} of a cubic 3-connected planar graph $G$ with 13 $\Pi$-unstable edges forming three $\Pi$-unstable components. The $\Pi$-unstable dual, indicated in the figure by square vertices and green solid edges, contains three cycles, two of which share a single vertex. 
Furthermore, the $\Pi$-unstable dual is an zero-homologous subgraph of the dual graph. We show that this is true for any unstable dual and fully characterize the set of all unstable duals of a graph. 

\begin{proposition}
\label{prop:eulerian}
    Let $G$ be a cubic graph with base embedding $\hatPi$ and let $G^*$ be its dual graph. Let $H^*$ be a subgraph of $G^*$ without isolated vertices. Then $H^*$ is the $\Pi$-unstable dual $G^*_\Pi$ for some embedding $\Pi$ of $G$ if and only if $H^*$ is the empty graph or $H^*$ is zero-homologous on the surface of the base embedding $\hatPi$. 
    \end{proposition}

\begin{proof}
    ($\Leftarrow$) Suppose $H^*$ is a subgraph (not necessarily induced) of $G^*$ that is zero-homologous on the surface of the base embedding $\hatPi$. Then, the edges of $H^*$ split the surface into two parts, $A\cup B$, with $\partial A = \partial B = H^*$
    
    Let $U$ be the set of vertices of $G$ in $A$ and let $V$ be the set of vertices in $B$. If $|U|\leq \frac{n}{2}$, let $\Pi$ be the embedding of $G$ with $\Pi$-unstable set $\Unst_\Pi=U$. Otherwise, let $\Pi$ be the embedding of $G$ with $\Unst_\Pi=V$. 

    Now, consider an edge $e$ with one endpoint in $U$ and the other in $V$. This is a $\Pi$-unstable edge and thus its dual edge $e^*$ is contained in $G^*_\Pi$. It is also clear that $e^* \in E(H^*)$ and that every edge of $H^*$ is of this form. Thus, $H^*$ and $G^*_\Pi$ have the same set of edges and so $H^*$ is the $\Pi$-unstable dual of the embedding $\Pi$ of $G$. 

    Finally, if $H^*$ is the empty graph, then it is the unstable dual of the base embedding $\hatPi$. 

    ($\Rightarrow$) Let $H^*$ be the $\Pi$-unstable dual $G^*_\Pi$ for some embedding $\Pi$ of $G$. If $H^*$ is nonempty, the set of $\Pi$-unstable vertices in $G$, $U_\Pi$, and its complement must both be nonempty. By definition, the set of $\Pi$-unstable edges is an edge-cut, $\delta(U_\Pi)$, separating $U_\Pi$ from $V(G)\setminus U_\Pi$. In the dual graph, the edges in $\delta(U_\Pi)^*$ must separate the faces of $G^*$ corresponding to $\Pi$-unstable vertices from those corresponding to $\Pi$-stable vertices. In other words, the edges of $H^*$ separate the surface into two parts, one corresponding to the union of the $\Pi$-unstable components. Thus, $H^*$ is zero-homologous in the surface of the base embedding $\hatPi$. 
\end{proof}

We now examine the structure of the unstable dual of a cubic graph and determine properties that indicate the genus of the corresponding embedding of $G$. We begin by describing ways of reducing the unstable dual and the effects of these reductions on the genus of the corresponding embeddings. 

As a tool, we use the earlier described signature on the edges of a cubic graph corresponding to the $\Pi$-stability of each edge. Each $\Pi$-unstable edge receives a negative signature and each $\Pi$-stable edge receives a positive signature. Then, to trace the faces in the embedding $\Pi$, we can think of switching the direction at the next vertex whenever an edge with negative signature (i.e.~an unstable edge) is crossed. The subwalks of these facial walks consisting only of stable edges (those with positive signature) are the \emph{stable subwalks} of the facial walk. Then, a $\Pi$-facial walk is a sequence of stable subwalks, separated by $\Pi$-unstable edges. 

Let $\hatPi$ be the base embedding of a cubic graph $G$, and $\Pi$ an embedding of $G$. We define $\Delta g(\Pi)$ to be the difference in genus from $\hatPi$ to $\Pi$. More formally, $\Delta g(\Pi) = g(\Pi)-g(\hatPi) = \frac{1}{2}(f(\hatPi)-f(\Pi))$, where $f(\Pi)$ is the number of faces in the embedding $\Pi$. Note that $\Delta g(\Pi)$ may be positive or negative. 

\begin{proposition}
    Let $G$ be a cubic graph and $\Pi$ an embedding of $G$, with $\Pi$-unstable dual $G^*_\Pi$. Suppose that $G^*_\Pi$ is the disjoint union of two graphs $H_1$ and $H_2$, each of which is zero-homologous. Then there exist embeddings of $G$, $\Pi_1$ and $\Pi_2$, such that $G^*_{\Pi_1} = H_1$ and $G^*_{\Pi_2} = H_2$. Furthermore, $\Delta g(\Pi) = \Delta g(\Pi_1)+ \Delta g(\Pi_2)$.\label{prop:disconnected dual}
\end{proposition}

\begin{proof}
That $H_1$ and $H_2$ are the unstable duals of some embeddings $\Pi_1$ and $\Pi_2$, respectively, of $G$ follows directly from Proposition~\ref{prop:eulerian}. 

Now, we show that these embeddings satisfy the desired genus condition. Let $\hatPi$ be the base embedding of $G$. As discussed above, the facial walks of $\Pi$ each consist of a sequence of alternating $\Pi$-stable subwalks and $\Pi$-unstable edges. So each $\Pi$-face can only contain stable subwalks belonging to the $\hatPi$-faces in one of the two components of $G^*_\Pi$. 

If $f(\hatPi)$ is the number of faces in the base embedding and $f(\Pi)$ is the number of faces in the embedding $\Pi$, then $\Delta g(\Pi) = \frac{1}{2}\big(f(\hatPi)-f(\Pi)\big)$. Similarly, let $f(\Pi_1)$ and $f(\Pi_2)$ be the number of faces in $\Pi_1$ and $\Pi_2$, respectively. Then, $\Delta g(\Pi_1) = \frac{1}{2}\big(f(\hatPi)-f(\Pi_1)\big)$ and $\Delta g(\Pi_2) = \frac{1}{2}\big(f(\hatPi)-f(\Pi_2)\big)$. By the discussion above, we have $f(\hatPi)-f(\Pi) = \big(f(\hatPi)-f(\Pi_1)\big) + \big(f(\hatPi)-f(\Pi_2)\big)$ and therefore, $\Delta g(\Pi) = \Delta g(\Pi_1)+\Delta g(\Pi_2)$, as desired. 
\end{proof}

Proposition~\ref{prop:disconnected dual} implies that we can study the genus of the embeddings corresponding to only connected unstable duals. For disconnected unstable duals, we can determine the genus by considering their connected components. Next, we show that in fact we can also sum over 2-connected components of the unstable dual in a similar way. 

\begin{theorem}
    Let $G$ be a cubic graph and $\Pi$ an embedding of $G$ with unstable dual $G^*_\Pi$ that has a cut-vertex $f$. Suppose that  $G^*_\Pi$ is constructed from zero-homologous graphs $H_1,\dots,H_k$, where $V(H_i)\cap V(H_j) = \{f\}$ for all $1\le i < j \le k$. If the edges of each $H_i$ appear consecutively in the local rotation around $f$, then there exist embeddings $\Pi_1, \dots, \Pi_k$ of $G$ such that $G^*_{\Pi_i} = H_i$, for $i=1,\dots,k$. Furthermore, $\Delta g(\Pi) = \sum_{i=1}^{k} \Delta g(\Pi_i)$.
    \label{thm:cut vertex dual}
\end{theorem}

\begin{proof}
First, assume that $k=2$, so that $G^*_\Pi$ is constructed from zero-homologous graphs $H_1$ and $H_2$, whose intersection is the cut-vertex $f$. Since each $H_i$ is zero-homologous, by Proposition \ref{prop:eulerian}, there is some embedding $\Pi_i$ of $G$ such that $G^*_{\Pi_i}$ is the graph $H_i$, for $i=1,2$. 

Let $e_1$ and $e'_1$ be the first and last edges, respectively, in the local rotation at $f$ in $G^*_\Pi$ that belong to $H_1$. Define $e_2$ and $e'_2$ analogously for $H_2$. Then, in $\hatPi$, the boundary of the face $f$ has the form $e_1,\dots,e'_1,\dots,e_2,\dots,e'_2,\dots$. Construct the graph $G'$ from $G$ by subdividing the edges $e_2$ and $e'_2$ and joining the two new vertices with the edge $x_2$ inside the face $f$. This forms the base embedding $\hatPi'$ of $G'$. Note that $g(\hatPi') = g(\hatPi)$. 

Label the edges formed from the subdivided edges $e_2$ and $e'_2$ by $d_2$ and $e_2$ and by $e'_2$ and $d'_2$, respectively, as shown in Figure \ref{fig:cutvertexsplit}. The stability of the endpoints of $x_2$ should be such that the edges $e_2$ and $e'_2$ are $\Pi'$-unstable and the edges $d_2$ and $d'_2$ are $\Pi'$-stable, with respect to $\hatPi'$. Note that the edge $x_2$ must be $\Pi'$-stable. Now, the $\Pi'$-unstable dual of $G'$ is isomorphic to the disjoint union of $H_1$ and $H_2$. By Proposition \ref{prop:disconnected dual}, we have that $\Delta g(\Pi') = \Delta g(\Pi_1)+\Delta g(\Pi_2)$. 

\begin{figure}[ht]
    \centering
         \includegraphics[width=0.22\textwidth]{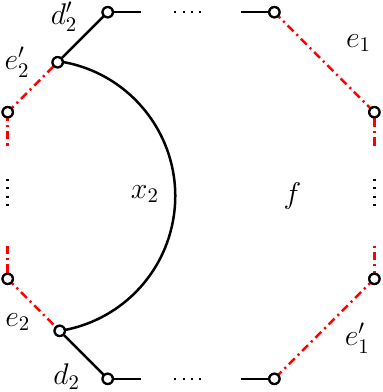}        
    \caption{Splitting the unstable face $f$ in $G$ by subdividing the unstable edges $e_2$ and $e'_2$ and adding an edge between the two new vertices.}
    \label{fig:cutvertexsplit}
\end{figure}

Finally, consider the following $\Pi$-facial walk in $G$. 
$$e_1,W'_1,e'_1,W_1,e_2,W'_2,e'_2,W_2,$$
where $W_1$ is the stable subwalk of the face $f$ between $e'_1$ and $e_2$, and $W'_1$ is the subwalk of this facial walk between $e_1$ and $e'_1$. The subwalks $W_2$ and $W'_2$ are defined analogously. Then, in $G'$, we have the $\Pi'$-facial walks
$$e_1,W'_1,e'_1,W_1,d_2,x_2,d'_2,W_2~~~\text{and}~~~e_2,W'_2,e'_2,x_2.$$
Thus, if the number of $\Pi$-faces in $G$ is $f(\Pi)$, the number of $\Pi'$-faces in $G'$ is $f(\Pi)+1$. Then, 
$$\Delta g(\Pi') = \frac{1}{2} \left( f(\hatPi')-f(\Pi')\right) = \frac{1}{2}\left(f(\hatPi)+1 - (f(\Pi)+1) \right) = \Delta g(\Pi).$$

The result for $k>2$ follows by induction on $k$ as follows. Define $H_2'$ as the union of $H_2,\dots,H_k$. Now, we use the case of $k=2$ for $H_1$ and $H_2'$, with the inductive hypothesis for $H_2'$. 
\end{proof}

Finally, we consider an unstable dual graph formed by identifying two smaller unstable duals at a 2-vertex-cut. Unlike in the case of identifying a single vertex, we must examine the facial structure to determine how the genus of the embedding can be computed using the genera of the embeddings corresponding to each of the smaller components. To do this, we require knowledge of whether a pair of stable subwalks in a particular embedding belong to the same facial walk. 

Let $G$ be a cubic graph and $\Pi$ an embedding of $G$ with unstable dual $G^*_\Pi$. Consider vertices $u$ and $v$ in $G^*_\Pi$, corresponding to faces $u$ and $v$ in the base embedding of $G$ that are $\Pi$-unstable. Consider a stable subwalk $U$ of $u$ and a stable subwalk $V$ of $v$. If both of these subwalks belong to the same $\Pi$-face, then we let $s_{\Pi}(U,V)=1$. Otherwise, we let $s_{\Pi}(U,V)=0$. 

For graphs $G_1$ and $G_2$, let $\phi(G_1,G_2)$ be the graph formed from the disjoint union of $G_1$ and $G_2$ by identifying a vertex $u_1\in V(G_1)$ with a vertex $u_2 \in V(G_2)$ and a vertex $v_1\in V(G_1)$ with a vertex $v_2\in V(G_2)$ in such a way that the edges of $G_1$ appear consecutively in the local rotation in $\phi(G_1,G_2)$ at both vertices. Suppose there exists a graph $G$ and embeddings $\Pi_1$ and $\Pi_2$ of $G$ such that $G^*_{\Pi_i} = G_i$, for $i=1,2$. Then, we let $U^\phi_1$ be the stable subwalk of the face $u_1$ in $\Pi_1$ containing the edges incident with $u_2$ in $V(G_2)$ and $U^\phi_2$ be the stable subwalk of $u_2$ in $\Pi_2$ containing the edges incident with $u_1$ in $V(G_1)$. Define $V^\phi_1$ and $V^\phi_2$ analogously for $v_1$ and $v_2$. 
Using this notation, we can determine the effect on the genus of the corresponding embedding when two unstable duals of a graph are joined at a 2-vertex-cut. 

\begin{theorem}
 \label{thm:2-cuts}
    Let $G$ be a cubic graph and $\Pi$ an embedding of $G$ with unstable dual $G^*_\Pi$. Suppose that $G^*_\Pi = \phi(G_1,G_2)$ for zero-homologous graphs $G_1$ and $G_2$. Then, there exist embeddings of $G$, $\Pi_1$ and $\Pi_2$, such that $G^*_{\Pi_i}$ is isomorphic to the graph $G_i$, for $i=1,2$.
    
    In addition, define $U^\phi_i$ and $V^\phi_i$ as above, for $i=1,2$. Then  $$\Delta g(\Pi) = \Delta g(\Pi_1)+\Delta g(\Pi_2)-s_{\Pi_1}(U^\phi_1,V^\phi_1)s_{\Pi_2}(U^\phi_2,V^\phi_2).$$
   \end{theorem}

\begin{proof}
First, note that by Proposition \ref{prop:eulerian}, there exist embeddings, $\Pi_1$ and $\Pi_2$, of $G$ such that $G^*_{\Pi_i}$ is the graph $G_i$, for $i=1,2$. 

Now, if $f(\hatPi)$ is the number of faces in the base embedding and $f(\Pi)$, $f(\Pi_1)$ and $f(\Pi_2)$ are the number of faces in $\Pi$, $\Pi_1$ and $\Pi_2$, respectively, then $\Delta g(\Pi_1) = \frac{1}{2}(f(\hatPi)-f(\Pi_1))$, $\Delta g(\Pi_2) = \frac{1}{2}(f(\hatPi)-f(\Pi_2))$, and $\Delta g(\Pi) = \frac{1}{2}(f(\hatPi)-f(\Pi))$. Let $f_s$ be the number of $\Pi$-stable faces in the base embedding of $G$. These faces correspond to the vertices of $G^*$ that do not appear as vertices of $G^*_\Pi$. Then, these faces must also be $\Pi_1$-stable and $\Pi_2$-stable, since they do not appear as vertices in $G_1$ or in $G_2$. Note that $|V(G_1)|+|V(G_2)|-2=|V(G^*_\Pi)|$. Then, we can rewrite each of the above equations as $\Delta g(\Pi_1) = \frac{1}{2}(f_s+|V(G_1)|+|V(G_2)|-2-f(\Pi_1))$, $\Delta g(\Pi_2) = \frac{1}{2}(f_s+|V(G_1)|+|V(G_2)|-2-f(\Pi_2))$, and $\Delta g(\Pi) = \frac{1}{2}(f_s+|V(G^*_\Pi)|-f(\Pi))$. We consider each of the three cases for $s_{\Pi_1}(U^\phi_1,V^\phi_1)$ and $s_{\Pi_2}(U^\phi_2,V^\phi_2)$. 

First, suppose that $s_{\Pi_1}(U^\phi_1,V^\phi_1)=s_{\Pi_2}(U^\phi_2,V^\phi_2)=0$, so $U^\phi_1$ and $V^\phi_1$ belong to different $\Pi_1$-faces, and $U^\phi_2$ and $V^\phi_2$ belong to different $\Pi_2$-faces. In $\Pi_1$, we have the facial walks 
$$e_1,U^\phi_1,e_2,W_1 \text{~~and~~}
f_1,V^\phi_1,f_2,W_2$$
where $e_1,e_2$ are edges in the $\hatPi$-face $u$, and $f_1,f_2$ are edges in the $\hatPi$-face $v$. Similarly, in $\Pi_2$, we have the facial walks 
$$h_1,U^\phi_2,h_2,W_3 \text{~~and~~}
k_1,V^\phi_2,k_2,W_4$$ where $h_1,h_2$ are edges in $u$ and $k_1,k_2$ are edges in $v$. Let $U^\phi_{1,2}$ be the subwalk of $U^\phi_1\cap U^\phi_2$ between the edges $e_1$ and $h_2$, and $U^\phi_{2,1}$ the subwalk of $U^\phi_1\cap U^\phi_2$ between the edges $h_1$ and $e_2$. Define $V^\phi_{1,2}$ and $V^\phi_{2,1}$ similarly. Then, in $\Pi$, we have the facial walks
$$e_1,U^\phi_{1,2},h_2,W_3,h_1,U^\phi_{2,1},e_2,W_1 \text{~~and~~}
f_1,V^\phi_{1,2},k_2,W_4,k_1,V^\phi_{2,1},f_2,W_2.$$
See Figure~\ref{fig:2cut}. 
The other facial walks in $\Pi_1$ or $\Pi_2$ remain facial walks in $\Pi$, including the stable facial walks. Thus, $f(\Pi) = f(\Pi_1)+f(\Pi_2)-2-(f_s+|V(G_1)|-2+|V(G_2)|-2)$, where the latter terms count the faces that are $\Pi_1$-stable or $\Pi_2$-stable to avoid double-counting. Then, using the equations above, we have 
$$\Delta g(\Pi) = \frac{1}{2}(f_s + |V(G^*_\Pi)|-f(\Pi))$$
$$ = \frac{1}{2}(f_s+|V(G_1)|+|V(G_2)|-2-\left(f(\Pi_1)+f(\Pi_2)-2-(f_s+|V(G_1)|-2+|V(G_2)|-2)\right))$$
$$ = \frac{1}{2}(f_s+|V(G_1)|+|V(G_2)|-2-f(\Pi_1) + f_s+|V(G_1)|+|V(G_2)|-2-f(\Pi_2))$$
$$ = \Delta g(\Pi_1)+\Delta g(\Pi_2).$$

 \begin{figure}
    \centering
    \begin{subfigure}[b]{0.45\textwidth}
         \centering
         \includegraphics[width=0.75\textwidth]{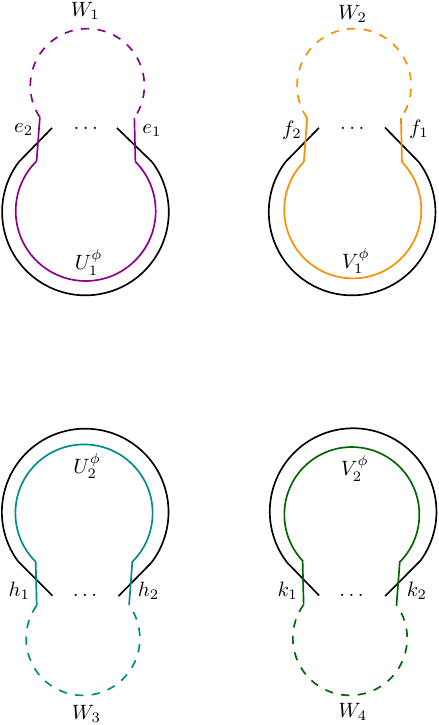}
         \caption{Facial walks in $\Pi_1$ (top) and $\Pi_2$ (bottom)}
         \label{fig:2cut1}
     \end{subfigure}
     \begin{subfigure}[b]{0.45\textwidth}
         \centering
         \includegraphics[width=0.75\textwidth]{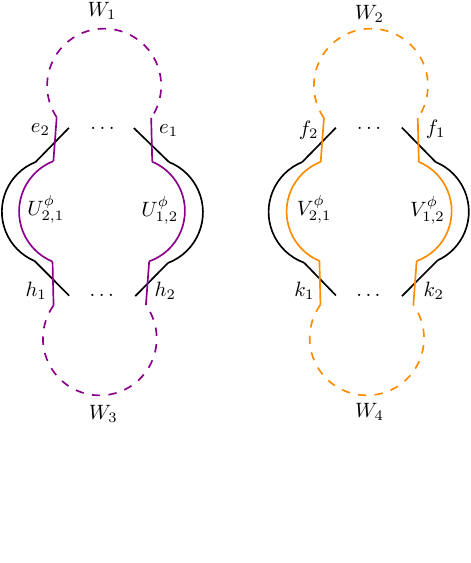}
         \caption{Facial walks in $\Pi$}
         \label{fig:2cut2}
     \end{subfigure}
    \caption{Two distinct facial walks in each of $\Pi_1$ and $\Pi_2$ become two facial walks in $\Pi$.}
    \label{fig:2cut}
\end{figure}

Now, suppose without loss of generality that $s_{\Pi_1}(U^\phi_1,V^\phi_1)=1$ and $s_{\Pi_2}(U^\phi_2,V^\phi_2)=0$. In $\Pi_1$, we have the facial walk 
$$e_1,U^\phi_1,e_2,W_1,f_1,V^\phi_1,f_2,W_2$$
where $e_1,e_2$ are edges in the $\hatPi$-face $u$, and $f_1,f_2$ are edges in the $\hatPi$-face $v$. In $\Pi_2$, we have the facial walks 
$$h_1,U^\phi_2,h_2,W_3 \text{~~and~~}
k_1,V^\phi_2,k_2,W_4$$ where $h_1,h_2$ are edges in $u$ and $k_1,k_2$ are edges in $v$. Define $U^\phi_{1,2}$, $U^\phi_{2,1}$, $V^\phi_{1,2}$ and $V^\phi_{2,1}$ as above. Then, in $\Pi$, we have the facial walk
$$e_1,U^\phi_{1,2},h_2,W_3,h_1,U^\phi_{2,1},e_2,W_1,f_1,V^\phi_{1,2},k_2,W_4,k_1,V^\phi_{2,1},f_2,W_2.$$
The other facial walks in $\Pi_1$ or $\Pi_2$ remain facial walks in $\Pi$, including the stable facial walks. Thus, $f(\Pi) = f(\Pi_1)+f(\Pi_2)-2-(f_s+|V(G_1)|-2+|V(G_2)|-2)$ and, using the same calculations as above, we have 
$$ \Delta g(\Pi)= \Delta g(\Pi_1)+\Delta g(\Pi_2).$$

Finally, suppose that $s_{\Pi_1}(U^\phi_1,V^\phi_1)=s_{\Pi_2}(U^\phi_2,V^\phi_2)=1$. In $\Pi_1$, we have the facial walk 
$$e_1,U^\phi_1,e_2,W_1,f_1,V^\phi_1,f_2,W_2$$
where $e_1,e_2$ are edges in the $\hatPi$-face $u$, and $f_1,f_2$ are edges in the $\hatPi$-face $v$. In $\Pi_2$, we have the facial walk
$$h_1,U^\phi_2,h_2,W_3,k_1,V^\phi_2,k_2,W_4$$ where $h_1,h_2$ are edges in $u$ and $k_1,k_2$ are edges in $v$. Define $U^\phi_{1,2}$, $U^\phi_{2,1}$, $V^\phi_{1,2}$ and $V^\phi_{2,1}$ as above. Then, in $\Pi$, we have the facial walks
$$e_1,U^\phi_{1,2},h_2,W_3,k_1,V^\phi_{2,1},f_2,W_2 \text{~~and~~}
h_1,U^\phi_{2,1},e_2,W_1,f_1,V^\phi_{1,2},k_2,W_4.$$
The other facial walks in $\Pi_1$ or $\Pi_2$ remain facial walks in $\Pi$, including the stable facial walks. Thus, $f(\Pi) = f(\Pi_1)+f(\Pi_2)-(f_s+|V(G_1)|-2+|V(G_2)|-2)$. Note the difference of two from the previous cases. Then, using the equations above, we have 
$$\Delta g(\Pi) = \frac{1}{2}(f_s + |V(G^*_\Pi)|-f(\Pi))$$
$$ = \frac{1}{2}(f_s+|V(G_1)|+|V(G_2)|-2-\left(f(\Pi_1)+f(\Pi_2)-2-(f_s+|V(G_1)|+|V(G_2)|-2)\right))$$
$$ = \frac{1}{2}(f_s+|V(G_1)|+|V(G_2)|-2-f(\Pi_1) + f_s+|V(G_1)|+|V(G_2)|-2-f(\Pi_2) - 2)$$
$$ = \Delta g(\Pi_1)+\Delta g(\Pi_2) - 1.$$

In general, these cases show that $\Delta g(\Pi) = \Delta g(\Pi_1)+\Delta g(\Pi_2)-s_{\Pi_1}(U^\phi_1,V^\phi_1)s_{\Pi_2}(U^\phi_2,V^\phi_2)$. 
\end{proof}

Consider the unstable dual graph $G^*_\Pi$ shown in Figure \ref{fig:dualcutex1} and let $G$ be a cubic 3-connected planar graph. The unstable dual graph can be constructed from the two graphs shown in Figure \ref{fig:dualcutex2} using the operation $\phi$ described above. Let $\Pi_1$ be the embedding corresponding to the component on the left, and $\Pi_2$ the embedding corresponding to the component on the right. Then, by Theorem \ref{thm:2-cuts}, $\Delta g(\Pi)=\Delta g(\Pi_1)+\Delta g(\Pi_2)-s_{\Pi_1}(U^\phi_1,V^\phi_1)s_{\Pi_2}(U^\phi_2,V^\phi_2)$. The graph $G^*_{\Pi_1}$ can be further split at the cut vertex into two copies of $K_3$, corresponding to embeddings $\Pi_{1,1}$ and $\Pi_{1,2}$ of $G$. Each of these corresponds to an unstable cut of size 3, so by Corollary \ref{cor:g1cut}, $\Delta g(\Pi_{1,1}) = \Delta g(\Pi_{1,2})=1$. Then, by Theorem \ref{thm:cut vertex dual}, we have that $\Delta g(\Pi_1) = \Delta g(\Pi_{1,1})+\Delta g(\Pi_{1,2})=2$. By tracing the appropriate faces in the corresponding embeddings, we can see that $s_{\Pi_1}(U^\phi_1,V^\phi_1)=s_{\Pi_2}(U^\phi_2,V^\phi_2)=1$. Therefore, $\Delta g(\Pi) = 2+\Delta g(\Pi_2)-1 = 2+1-1=2$. 

 \begin{figure}[!h]
    \centering
    \begin{subfigure}[b]{0.45\textwidth}
         \centering
         \includegraphics[width=0.65\textwidth]{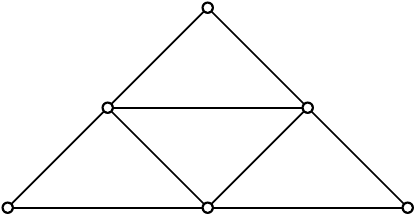}
         \caption{$G^*_\Pi$}
         \label{fig:dualcutex1}
     \end{subfigure}
     \begin{subfigure}[b]{0.45\textwidth}
         \centering
         \includegraphics[width=0.85\textwidth]{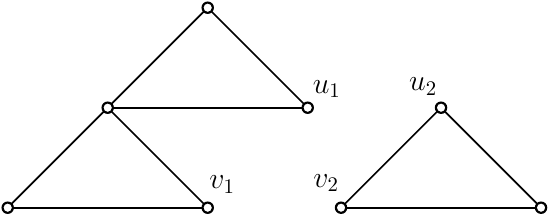}
         \caption{$G^*_{\Pi_1}$ and $G_{\Pi_2}$}
         \label{fig:dualcutex2}
     \end{subfigure}
     \begin{subfigure}[b]{0.45\textwidth}
         \centering
         \includegraphics[width=0.85\textwidth]{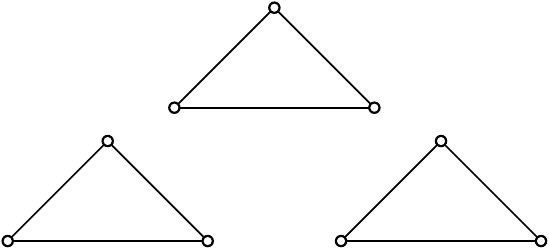}
         \caption{$G_{\Pi_{1,1}}$, $G_{\Pi_{1,2}}$ and $G_{\Pi_2}$}
         \label{fig:dualcutex3}
     \end{subfigure}
    \caption{Splitting $G^*_\Pi$ at a 2-cut and again at a cut vertex to determine the genus of the corresponding embedding. }
    \label{fig:dualcutex}
\end{figure}

The dual graph of a cubic graph is a simple graph when the base embedding is polyhedral. However, loops\footnote{Here and throughout the paper we consider the term \emph{loop} in the sense of graph theory as an edge incident with a single vertex.} and double edges may be present in the dual graph of a cubic graph whose base embedding is not polyhedral.

\begin{proposition}
\label{prop:loops}Let $G$ be a cubic graph with base embedding $\hatPi$, and let $\Pi$ be another embedding of $G$. Let $G^*_\Pi$ be the $\Pi$-unstable dual of $G$. If $G^*_\Pi$ contains a zero-homologous loop $e^*$ at a vertex $f\in V(G^*_\Pi)$, then there exists an embedding $\Pi'$ of $G$ such that $g(\Pi) = g(\Pi')$ and  
    \[ G^*_{\Pi'} = \begin{cases} 
      G^*_\Pi-f & \text{if~}d_{G^*_\Pi}(f)=2 \\
      G^*_\Pi-e^* & \text{otherwise.} 
   \end{cases}
\] 
\end{proposition}

\begin{proof}
    From $\Pi$, we construct a new embedding $\Pi'$ and show that it satisfies the required conditions. The loop $e^*\in E(G^*_\Pi)$ corresponds to a $\Pi$-unstable edge, $e\in E(G)$, that appears twice on the same face in the base embedding $\widehat{\Pi}$.

    Since $e^*$ is zero-homologous, $e$ must be a cut-edge in $G$, with one $\Pi$-stable endpoint, $u_1$, and one $\Pi$-unstable endpoint, $u_2$. Let $U_1$ and $U_2$ be the two components of $G-e$, with $u_1\in V(U_1)$ and $u_2\in V(U_2)$. Now, let $\Pi'$ be the embedding of $G$ with $\pi'_v = \pi_v$ if $v\in V(U_1)$ and $\pi'_v = \pi^{-1}_v$ if $v\in V(U_2)$. Then, every $\Pi$-unstable vertex in $U_2$ is $\Pi'$-stable, and vice versa. The stability of each vertex in $U_1$ is unchanged. 

    The only edge connecting $U_1$ and $U_2$ in $G$ is the edge $e$, which has two $\Pi'$-stable endpoints and is thus $\Pi'$-stable. The $\Pi'$-stability of any edge with both endpoints in $U_1$ or both endpoints in $U_2$ is the same as the $\Pi$-stability. Therefore, if $\delta(U_\Pi)$ and $\delta(U_{\Pi'})$ are the sets of $\Pi$-unstable edges and $\Pi'$-unstable edges, respectively, then we have $\delta(U_{\Pi'}) = \delta(U_\Pi)\setminus\{e\}$. 

    We now show that $G^*_{\Pi'}$ satisfies the desired conditions. If $d_{G^*_\Pi}(f)=2$, then $e$ must be the only $\Pi$-unstable edge on the face $f$ in $G$. This means that in the embedding $\Pi'$, the face $f$ has no $\Pi'$-unstable edges on its boundary and is therefore $\Pi'$-stable. Thus, $f\not\in V(G^*_{\Pi'})$. Therefore, we have $E(G^*_{\Pi'}) = E(G^*_\Pi)\setminus\{e^*\}$ and $V(G^*_{\Pi'}) = V(G^*_\Pi)\setminus\{f\}$. 

    If $d_{G^*_\Pi}(f)>2$, then the face $f$ must contain at least one other $\Pi$-unstable edge, which is also $\Pi'$-unstable. So $f$ is a $\Pi'$-unstable face and appears as a vertex of $G^*_{\Pi'}$. In this case, we have $E(G^*_{\Pi'})=E(G^*_\Pi)\setminus\{e^*\}$, and $V(G^*_{\Pi'})=V(G^*_\Pi)$. 

    Therefore, $G^*_{\Pi'}$ is obtained from $G^*_\Pi$ by removing $e^*$ (and removing the incident vertex $f$ if $d_{G^*_\Pi}(f)=2$), as claimed. 

     Finally, we show that $g(\Pi)=g(\Pi')$, by showing that the two embeddings have the same number of faces. First, we note that if a face $F$ in $\Pi$ does not contain the edge $e$, then it is also a face in $\Pi'$. 
     
     \begin{figure}
    \centering
    \begin{subfigure}[b]{0.48\textwidth}
         \centering
         \includegraphics[width=0.9\textwidth]{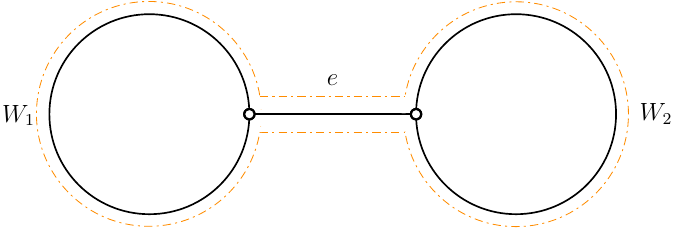}
         \caption{The face $f'$ in $\Pi'$}
         \label{fig:dualloop1}
     \end{subfigure}
\hfill
     \begin{subfigure}[b]{0.48\textwidth}
         \centering
         \includegraphics[width=0.9\textwidth]{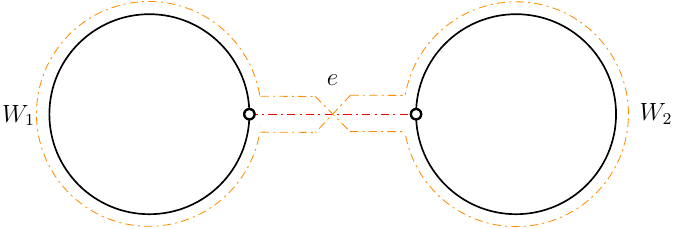}
         \caption{The face $f$ in $\Pi$}
         \label{fig:dualloop2}
     \end{subfigure}
    \caption{The effect of changing the stability at an edge in $G$ corresponding to a loop in $G^*$}
    \label{fig:dualloop}
\end{figure}

     Now, consider the face $f'$ in $\Pi'$ that contains the edge $e$. Since the edge $e$ appears twice in the boundary of $f'$, we can express $f'$ as $eW_1eW_2$, where $W_1$ and $W_2$ are the subwalks of the facial walk between appearances of $e$, shown in Figure \ref{fig:dualloop1}. Denote by $W_1^{-1}$ and $W_2^{-1}$ the subwalks $W_1$ and $W_2$ traversed in the opposite direction. Then, by changing the stability of only the edge $e$, we obtain the facial walk $f$: $eW_1^{-1}eW_2$, shown in Figure \ref{fig:dualloop2}. Thus, the face $f'$ in $\Pi'$ becomes one face $f$ in $\Pi$. Since the number of faces is the same in $\Pi$ as in $\Pi'$, the genus of these two embeddings must be the same. 
\end{proof}

 We have that an edge $e^*$ in the unstable dual graph $G^*_\Pi$ is a zero-homologous loop if and only if the corresponding edge $e$ in $G$ is a cut-edge. Thus, we may consider only 2-connected graphs in the remainder of the section. 

 Next, we discuss 2-cuts in $G$. Recall that edges $e_1,e_2$ form a 2-cut in $G$ if and only if their dual edges $e_1^*,e_2^*$ form a zero-homologous digon in $G^*$. We show that removing a zero-homologous digon from the unstable dual also has no effect on the genus of the corresponding embedding. 

\begin{proposition}
    Let $G$ be a cubic graph with base embedding $\hatPi$ and let $\Pi$ be another embedding of $G$. Let $G^*_\Pi$ be the $\Pi$-unstable dual of $G$. Suppose that $G^*_\Pi$ contains a zero-homologous digon with vertices $f$ and $f_1$. 
    Then there exists an embedding $\Pi'$ of $G$ such that $g(\Pi)=g(\Pi')$ and $G^*_{\Pi'}$ is obtained from $G^*_\Pi$ by removing the edges of the digon (and removing $f$ or $f_1$ if it becomes of degree zero).
 \label{prop:digons}
\end{proposition}

\begin{proof}
Let $e^*_1$ and $e^*_2$ be the two edges joining $f$ and $f_1$ in $G^*_\Pi$. 
Now, from $\Pi$, we construct a new embedding $\Pi'$ and show that it satisfies the required conditions. The parallel edges $e^*_1$ and $e^*_2$ correspond to a pair of $\Pi$-unstable edges, $e_1$ and $e_2$, with respect to $\widehat{\Pi}$ that both lie on the two faces $f$ and $f_1$ in $\hatPi$. 

\begin{figure}[ht]
    \centering
         \includegraphics[width=0.35\textwidth]{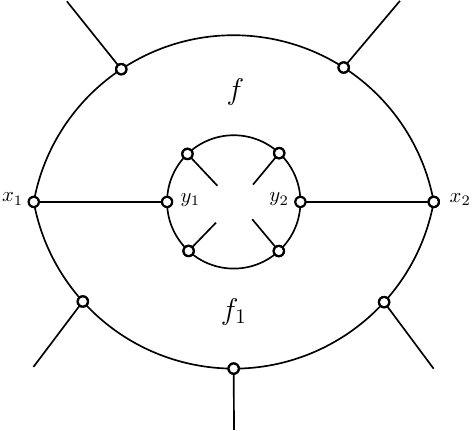}        
    \caption{The structure and labelling of two faces that share a pair of $\Pi$-unstable edges in the proof of Theorem~\ref{prop:digons}. }
    \label{fig:digon}
\end{figure}

Since the pair of dual parallel edges are zero-homologous, $e_1$ and $e_2$ must form a 2-cut in $G$. Let $X$ and $Y$ be the two components of $G-\{e_1,e_2\}$. Label the endpoints $e_1=x_1y_1$ and $e_2=x_2y_2$ with $x_1,x_2\in V(X)$ and $y_1,y_2\in V(Y)$, as shown in Figure \ref{fig:digon}. Now, let $\Pi'$ be the embedding of $G$ with $\pi'_v = \pi_v$ if $v\in V(X)$ and $\pi'_v = \pi_v^{-1}$ if $v\in V(Y)$. Then, every $\Pi$-unstable vertex in $Y$ is $\Pi'$-stable, and vice versa. The stability of each vertex in $X$ is unchanged. 

There are only two edges connecting $X$ and $Y$ in $G$: the edges $e_1$ and $e_2$, which are $\Pi'$ stable, since the stability of exactly one endpoint of each edge was changed. The $\Pi'$-stability of any edge with both endpoints in $X$ or both endpoints in $Y$ is the same as the $\Pi$-stability. Thus, if $\delta(U_\Pi)$ and $\delta(U_{\Pi'})$ are the sets of $\Pi$-unstable edges and $\Pi'$-unstable edges, respectively, then we have $\delta(U_{\Pi'}) = \delta(U_{\Pi})\setminus\{e_1,e_2\}$. 

Now, we show that $G^*_{\Pi'}$ satisfies the desired conditions. Note that if $d_{G^*_\Pi}(f) = 2$, the edges $e_1$ and $e_2$ are the only $\Pi$-unstable edges on the face $f$ in $G$. If $d_{G^*_\Pi}(f_1)=2$, then these two edges are also the only $\Pi$-unstable edges on the face $f_1$ in $G$. This means that in the embedding $\Pi'$, the faces $f$ and $f_1$ both have no $\Pi$-unstable edges on their boundary and are $\Pi'$-stable. Then, $f,f_1\not\in V(G^*_{\Pi'})$. Therefore, we have $E(G^*_{\Pi'}) = E(G^*_\Pi)\setminus\{e^*_1,e^*_2\}$ and thus $V(G^*_{\Pi'}) = V(G^*_\Pi)\setminus\{f,f_1\}$. 

Similarly, if exactly one of $f$ and $f_1$, say $f$, is of degree two in $G^*_\Pi$, then we have $E(G^*_{\Pi'}) = E(G^*_\Pi)\setminus\{e^*_1,e^*_2\}$ and thus $V(G^*_{\Pi'}) = V(G^*_\Pi)\setminus\{f\}$. 

If $d_{G^*_\Pi}(f)>2$, then the face $f$ must contain at least one other $\Pi$-unstable edge, which is also $\Pi'$-unstable. So $f$ is a $\Pi'$-unstable face and appears as a vertex of $G^*_{\Pi'}$. Similarly if $d^*_{G^*_\Pi}(f_1)>2$. In this case, we have $E(G^*_{\Pi'}) = E(G^*_\Pi)\setminus\{e^*_1,e^*_2\}$ and $V(G^*_{\Pi'}) = V(G^*_\Pi)$. 

Therefore, $G^*_{\Pi'}$ is obtained from $G^*_\Pi$ by removing $e_1^*$ and $e_2^*$ (and removing the incident vertex $f$ or $f_1$ if it becomes of degree zero), as claimed.  

Finally, we show that $g(\Pi) = g(\Pi')$, by showing that the two embeddings have the same number of faces. First, we note that if a face $F$ in $\Pi$ does not contain either of the edges $e_1$ or $e_2$, then it is also a face in $\Pi'$. 

Now, consider the faces $f'$ and $f_1'$ in $\Pi'$ that contain the edges $e_1$ and $e_2$. We can express $f'$ as $e_1,W_1,e_2,W_2$ and $f_1'$ as $e_1,W_3,e_2,W_4$, where $W_1,\dots,W_4$ are the subwalks of $f'$ and $f_1'$ between appearances of $e_1$ and $e_2$. Denote by $W_i^{-1}$ the subwalk $W_i$ traversed in the opposite direction. Then, by changing the stability of only the edges $e_1$ and $e_2$, we obtain the facial walks $e_1,W_4^{-1},e_2,W_2$ and $e_1,W_3,e_2,W_1^{-1}$. Thus, the two faces $f'$ and $f_1'$ become two faces in $\Pi$. Since the number of faces is the same in $\Pi$ as in $\Pi'$, the genus of these two embeddings is the same. 
\end{proof}

We note the correspondence between this method of removing a pair of parallel edges in the unstable dual of a cubic graph and Whitney's 2-switching theorem \cite{whitney2}. 

This gives a method of reduction for determining the genus of an embedding $\Pi$ given its corresponding unstable dual with respect to some base embedding. We can delete the zero-homologous loops and pairs of parallel edges with one endpoint of degree exactly two (and any resulting isolated vertices) without changing the genus of the embedding. 
However, if a loop is subdivided more than once, the genus increases by one with every two subdivisions.

\begin{theorem}
\label{cor:cycles}
Let $G$ be a cubic graph and $\Pi$ an embedding of $G$ whose unstable dual, $G^*_\Pi$, is the cycle graph $C_k$. If $G^*_\Pi$ is zero-homologous on the surface of the base embedding $\hatPi$, then $\Delta g(\Pi) = \left\lfloor \frac{k-1}{2} \right\rfloor$.
\end{theorem}

\begin{proof}
The cases $k=1$ and $k=2$ follow from Propositions \ref{prop:loops} and \ref{prop:digons}, respectively. For $k\geq 3$, we prove the result by induction. Let $C_k$ be the unstable dual of some embedding $\Pi$ of $G$, for $k\geq 3$. Let $f$ be some vertex of the unstable dual, corresponding to the $\Pi$-unstable face $f$ of $G$. Since $f$ has degree 2 in $G^*_\Pi$, the face $f$ contains two $\Pi$-unstable edges, $e_1$ and $e_2$. 

Now, construct the graph $H$ from $G$ by deleting the edges $e_1$ and $e_2$. To obtain a cubic graph, delete any resulting vertices of degree one and replace any resulting vertex of degree two with a diamond graph, shown in Figure \ref{fig:adddiamond}. The stability of all of the vertices in the diamond should match the stability of the vertex that it replaces. In the dual, this corresponds to contracting the two dual edges $e^*_1$ and $e^*_2$. 

\begin{figure}[ht]
    \centering
         \includegraphics[width=0.65\textwidth]{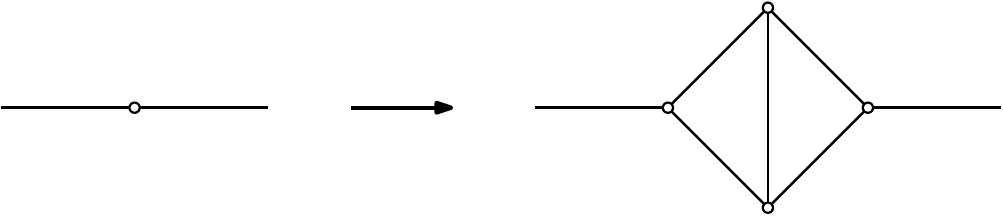}        
    \caption{In the construction of the graph $H$ in the proof of Theorem \ref{cor:cycles}, each vertex of degree 2 is replaced by a diamond. }
    \label{fig:adddiamond}
\end{figure}

Let $\Pi'$ be the embedding of $H$ where the only vertex stability changes from $\Pi$ are at the deleted vertices, since they are not vertices in $H$, and at the newly added vertices in the diamonds. However, the stability of these new vertices is such that the unstable edges incident with any degree two vertex remain unstable when the vertex is replaced. 

The corresponding base embedding $\hatPi_H$ of $H$ has $2-2c_2$ fewer faces than the base embedding $\hatPi$ of $G$, where $c_2$ is the number of vertices of degree two that were replaced by a diamond. When the $\Pi$-unstable edges on the face $f$ were deleted, the three $\hatPi$-faces $f$, $f_1$ and $f_2$ containing the edges $e_1$ and $e_2$ were merged to form a single face in $\hatPi_H$. Call this new face $f'$, and let $f_1$ be the face of $G$ containing $e_1$ and $f_2$ the face containing $e_2$. 

Next, we argue that $H^*_{\Pi'}$ is the cycle $C_{k-2}$. Each of the faces that are $\Pi$-unstable in $G$ and contain neither of the edges $e_1$ or $e_2$ are unchanged and thus remain $\Pi'$-unstable in $H$. Each of these faces contains two $\Pi$-unstable edges and thus has degree two in $G^*_\Pi$ and in $H^*_{\Pi'}$. The face $f_1$ in $G$ has two $\Pi$-unstable edges: $e_1$ and one other edge, say $e'$. If the edge $e'$ was deleted during the construction of $H$, then it must have been a $\Pi$-unstable edge in $G$ shared between the faces $f_1$ and $f_2$. In this case, $k=3$ and there are no remaining $\Pi'$-unstable faces. So $H^*_{\Pi'}$ is empty, having the same change in genus as an unstable dual with a single loop, by Proposition \ref{prop:loops}. Otherwise, $e'$ is shared with some $\Pi$-unstable face $f_3$ in $G$ that is neither $f$ or $f_2$, so it is also $\Pi'$-unstable in $H$. In $H$, the edge $e'$ is a $\Pi'$-unstable edge shared between $f'$ and $f_3$. 

Similarly, the face $f_2$ has exactly two $\Pi$-unstable edges: $e_2$ and $e''$. As above, if $e''$ was deleted in the construction of $H$, then $k=3$ and $H^*_{\Pi'}$ is empty. Otherwise, $e''$ is shared with some $\Pi$-unstable face $f_4$ in $G$ that is neither $f$ or $f_1$, so it is also $\Pi'$-unstable in $H$. Then, in $H$, the edge $e''$ is a $\Pi'$-unstable edge shared between $f'$ and $f_4$. 

Now, the face $f'$ contains only edges that were in the faces $f$, $f_1$ or $f_2$ in $G$. Since both $e_1$ and $e_2$ were deleted to form $H$, only $e'$ and $e''$ are $\Pi'$-unstable edges in $f'$. Thus, the vertex in $H^*_{\Pi'}$ corresponding to $f'$ has degree exactly two and its neighbors are the neighbors of $f_1$ and $f_2$ in $G^*_\Pi$ that are not $f^*$. Every other vertex in $H^*_{\Pi}$ has the same neighborhood as in $G^*_\Pi$. Thus, $H^*_{\Pi'} = C_{k-2}$. By the induction hypothesis, $\Delta g(\Pi') = \left\lfloor \frac{k-3}{2}\right\rfloor$. 

Now, we show that $\Delta g(\Pi) = \Delta g(\Pi')+1$ by tracing the faces in both $\Pi$ and $\Pi'$. If $k$ is even, then in $\Pi$, we have two facial walks containing edges in $f$, $f_1$ and $f_2$: 
$$e_1,W,e_2,W'_2,e'', \dots, e',W'_1 \text{~~and~~}
e_1,W',e_2,W_2,e'',\dots, e',W_1$$
where $W$ and $W'$ are stable subwalks in $f$ and $W_i$ and $W'_i$ are stable subwalks in $f_i$, for $i=1,2$. Then, in $\Pi'$, we also have two facial walks containing these edges:
$$W,W_2,e'',\dots, e',W_1 \text{~~and~~}
W',W'_2,e'', \dots, e',W'_1$$

Similarly, if $k$ is odd, then in $\Pi$ we have one facial walk containing edges in $f$, $f_1$ and $f_2$:
$$e_1,W,e_2,W'_2,e'',\dots, e',W_1,e_1,W',e_2,W_2,e'',\dots, e',W'_1$$
where $W, W', W_i$ and $W'_i$, for $i=1,2$, are as above. Then, in $\Pi'$, we also have one facial walk containing these edges:
$$W,W_2,e'', \dots, e',W'_1,W',W'_2,e'', \dots, e',W_1$$

Recall that $c_2$ is the number of vertices of degree 2 replaced with a diamond graph and note that $f(\Pi) = f(\Pi')-2c_2$. Then we have
$$\Delta g(\Pi) = \frac{1}{2}\left(f(\hatPi) - f(\Pi)\right) = \frac{1}{2}\left(f(\hatPi_H)+(2-2c_2)-(f(\Pi')-2c_2) \right) $$ $$= \frac{1}{2}\left(f(\hatPi_H) - f(\Pi') \right)+1 = \Delta g(\Pi')+1.$$
Therefore, $\Delta g(\Pi) = \left\lfloor \frac{k-3}{2} \right\rfloor + 1 = \left\lfloor \frac{k-1}{2}\right\rfloor$. \end{proof}

\section{Small genus embeddings of planar graphs}
\label{section:g2}

\subsection{Cut-edges}

Using the definition of the unstable dual graph, the terms $g_k(G)$ of the genus distribution can be computed by counting the number of zero-homologous subgraphs of the dual graph $G^*$ that correspond to an embedding with genus $k$ (i.e., embeddings with $\Delta g(\Pi) = k - g(\hatPi)$). Recall that the dual of any cut-edge in $G$ is always a zero-homologous loop in the surface of any embedding of $G$, and vice versa. 

For a cubic graph $G$, denote by $g'_k(G)$ the number of zero-homologous subgraphs $G^*_\Pi$ of the dual graph $G^*$ that contain no zero-homologous loops and such that $g(\Pi)=k$. It follows from Proposition \ref{prop:loops} that the term $g_k(G)$ of the genus distribution can be computed directly from $g'_k(G)$. 

\begin{proposition}
    Let $G$ be a cubic graph with base embedding $\hatPi$ and let $\ell$ be the number of cut-edges of $G$. Then, $g_k(G) = 2^\ell g'_k(G)$. 
\end{proposition} 

It follows that cut-edges in a cubic graph $G$ have no effect on the log-concavity of its genus distribution. This is an alternate view of the result of \cite{GrossFurst} that the genus distribution of a graph formed by joining two graphs $G$ and $H$ by a cut-edge $uv$, where $u\in V(G)$ and $v\in V(H)$, is a constant multiple of the convolution of the genus distributions of $G$ and $H$, where the constant is $d_G(u)d_H(v)$. 

\subsection{2-cuts, Whitney switching}

Next, we discuss effects of 2-cuts on the genus distribution, corresponding to a zero-homologous digon in the dual graph. 

Suppose that $\{e,f\}$ is a 2-cut in $G$. Then we write $e\sim f$. This defines a symmetric and transitive binary relation $\sim$ (the easy proof is left to the reader) and therefore all edges participating in 2-cuts in $G$ can be partitioned into equivalence classes $E_1,\dots,E_t$. In each $E_i$, any two edges form a 2-cut.

For each $i$, fix an edge $e_i\in E_i$.
Denote by $g''_k(G)$ the number of zero-homologous subgraphs $G^*_\Pi$ of the dual graph $G^*$ that contain no edges in $E_i\setminus\{e_i\}$ for $1\le i \le t$ and such that $g(\Pi)=k$. It follows from Proposition \ref{prop:digons} that the term $g_k(G)$ of the genus distribution can be computed directly from $g''_k(G)$. 

\begin{proposition}
    Let $G$ be a $2$-connected cubic graph with base embedding $\hatPi$ and let $E_1,\dots,E_t$ $(t\ge0)$ be the classes corresponding to the 2-cuts in $G$. Let $b_i=|E_i|$ for $1\le i\le t$. Fix an edge $e_i\in E_i$ for each $i$. Let 
    $\ell=b_1+\dots+b_t$. Then, $g_k(G) = 2^{\ell-t} g''_k(G)$. 
\end{proposition} 

Therefore, we may focus only on the unstable dual graphs of 3-connected graphs.

\subsection{3-connected planar graphs}

Using subgraphs of the dual graph, Enami \cite{enami} characterized the toroidal embeddings of a cubic 3-connected planar graph. 

\begin{theorem}
    [\cite{enami}] There exists a one-to-one correspondence between inequivalent embeddings of a cubic 3-connected planar graph $G$ on the torus and subgraphs of the dual graph $G^*$ isomorphic to $K_{2,2,2}$, $K_{2,2m}$ or $K_{1,1,2m-1}$ for some positive integer $m$. 
    \label{thm:enami1}
\end{theorem}
The proof of Enami is based on Mohar-Robertson theory of embeddings of planar graphs in nonplanar surfaces \cite{nonplanar,projective_plane}.
Using the unstable dual and results of Section \ref{section:cuts and duals}, we can obtain a simplified proof of Theorem \ref{thm:enami1}.  

We note that there does not appear to be a straightforward generalization of Enami's characterization to the unstable dual of an embedding $\Pi$ of a cubic graph with $\Delta g(\Pi) = 1$ whose base embedding is polyhedral. A nonplanar cubic graph may have many polyhedral embeddings in the same surface, as well as in surfaces of different genera \cite{polyhedral}. In the remainder of this section, we focus on cubic planar graphs. 

For a cubic planar graph that is also C5EC, only some of the subgraphs listed in Theorem \ref{thm:enami1} can appear in the dual graph $G^*$. Note that the dual of a cubic C5EC planar graph is a 5-connected planar triangulation. Using this information, for embeddings with genus 2, we can generalize Theorem \ref{thm:enami1} under stronger connectivity conditions.

First, we give a simple proof showing that a 5-connected planar graph can contain no copy of $K_{2,3}$. 

\begin{figure}[ht]
    \centering
         \includegraphics[width=0.25\textwidth]{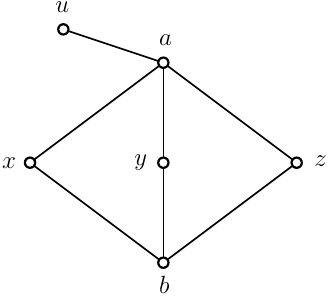}        
    \caption{A separating 4-cycle created by a copy of $K_{2,3}$.}
    \label{fig:k23}
\end{figure}

\begin{proposition}
    Let $G$ be a 5-connected planar graph. Then, $G$ contains no copy of $K_{2,3}$. 
    \label{prop:5connectedK23}
\end{proposition}

\begin{proof}
    Suppose that $G$ is 5-connected and contains a copy of $K_{2,3}$.  
    Note that we may assume $G$ is maximal planar. Since the smallest 5-connected maximal planar graph is the icosahedron, we may also assume $n\geq 12$. 

    Let $a$ and $b$ be distinct vertices in $G$ with three distinct common neighbors, $x$, $y$ and $z$, forming a copy of $K_{2,3}$. Since $G$ is 5-connected, $a$ must have a neighbor, say $u$, that is not in the set $\{b,x,y,z\}$. Assume, without loss of generality, that the edges incident with $a$ in a planar embedding of $G$ appear in the counterclockwise order $au, ax, ay, az$. Then, the cycle $axbz$ separates the vertex $u$ and the vertex $y$. See Figure \ref{fig:k23}. This separating 4-cycle contradicts $G$ being 5-connected and thus there can be no copy of $K_{2,3}$ in $G$. 
\end{proof}

Theorem \ref{thm:enami1} and Proposition \ref{prop:5connectedK23} implies that the only unstable duals corresponding to a toroidal embedding of a cubic C5EC planar graph are $K_{2,2}$ and $K_{1,1,1}$ or, equivalently, $C_4$ and $C_3$. No such subgraph can be separating in the dual graph. 

Now, using this characterization of embeddings with genus 1 and the results of Section \ref{section:cuts and duals}, we can construct unstable duals corresponding to embeddings of genus 2 of a cubic planar graph. Those that are connected and contain no copy of $K_{2,3}$ are shown in Figure~\ref{fig:g2duals}. In addition to these connected graphs, we have disconnected unstable duals corresponding to embeddings of genus 2, with exactly two components each of which is isomorphic to $K_{2,2,2}$, $K_{2,2m}$ or $K_{1,1,2m-1}$ for some positive integer $m$, by Theorem~\ref{thm:enami1}. 

 \begin{figure}[ht]
    \centering
    \hfill
    \begin{subfigure}[b]{0.2\textwidth}
         \centering
         \includegraphics[width=\textwidth]{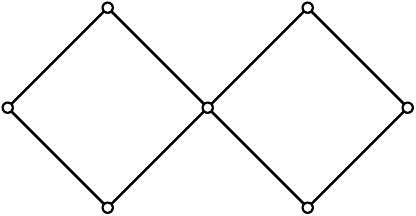}
         \caption{}
         \label{fig:g2-1}
     \end{subfigure}
\hfill
     \begin{subfigure}[b]{0.2\textwidth}
         \centering
         \includegraphics[width=0.88\textwidth]{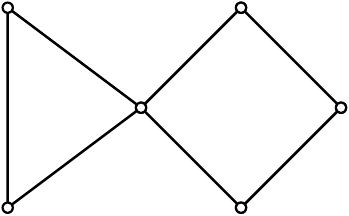}
         \caption{}
         \label{fig:g2-2}
     \end{subfigure}
\hfill
     \begin{subfigure}[b]{0.2\textwidth}
         \centering
         \includegraphics[width=0.73\textwidth]{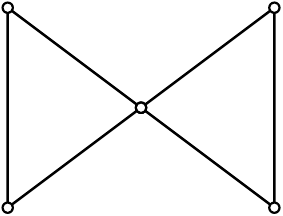}
         \caption{}
         \label{fig:g2-3}
     \end{subfigure}
     \hfill
     \begin{subfigure}[b]{0.2\textwidth}
         \centering
         \includegraphics[width=\textwidth]{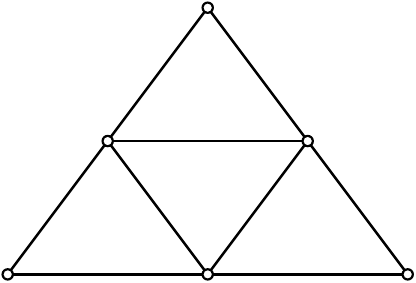}
         \caption{}
         \label{fig:g2-4}
     \end{subfigure}
     \hfill
     \hfill
     
     \hfill
    \begin{subfigure}[b]{0.15\textwidth}
         \centering
         \includegraphics[width=\textwidth]{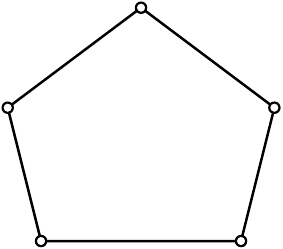}
         \caption{}
         \label{fig:g2-5}
     \end{subfigure}
     \hfill
     \begin{subfigure}[b]{0.15\textwidth}
         \centering
         \includegraphics[width=\textwidth]{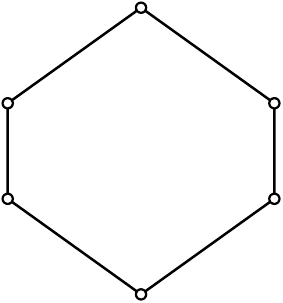}
         \caption{}
         \label{fig:g2-6}
     \end{subfigure}
\hfill
\begin{subfigure}[b]{0.15\textwidth}
         \centering
         \includegraphics[width=\textwidth]{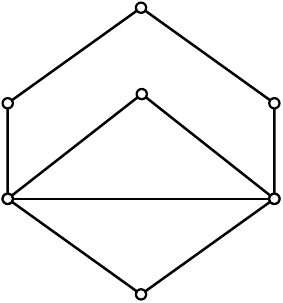}
         \caption{}
         \label{fig:g2-7}
     \end{subfigure}
     \hfill
     \vspace{0.5cm}
     \begin{subfigure}[b]{0.15\textwidth}
         \centering
         \includegraphics[width=\textwidth]{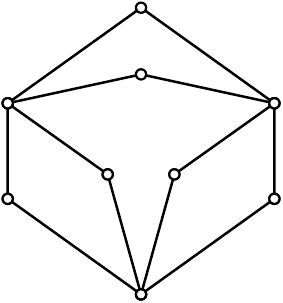}
         \caption{}
         \label{fig:g2-8}
     \end{subfigure}
\hfill\begin{subfigure}[b]{0.15\textwidth}
         \centering
         \includegraphics[width=\textwidth]{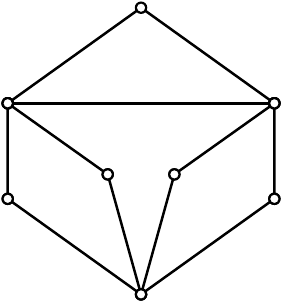}
         \caption{}
         \label{fig:g2-9}
     \end{subfigure}
     \hfill
     \hfill

    \caption{Connected unstable duals corresponding to an embedding of genus 2 in a cubic C5EC planar graph}
    \label{fig:g2duals}
\end{figure}

For cubic C5EC planar graphs, we are able to show that the unstable dual of any genus 2 embedding must be disconnected or be one of the nine graphs in Figure~\ref{fig:g2duals}. 

\begin{theorem}
\label{thm:g2duals}
Let $G$ be a cubic C5EC planar graph and $\Pi$ an embedding of $G$ with $g(\Pi)=2$. Then, $G^*_\Pi$ must either be the disjoint union of two 3-cycles, the disjoint union of a 3-cycle and a 4-cycle, the disjoint union of two 4-cycles, or one of the nine graphs shown in Figure~\ref{fig:g2duals}. 
\end{theorem}

\begin{proof}
We consider the possible structures of the $\Pi$-unstable faces and edges with respect to the planar embedding of $G$. Recall that, by Lemma \ref{lem:numCycles}, the maximum size of a peripheral family of $\Pi$-unstable cycles (i.e.~$\Pi$-unstable faces, since $G$ is C5EC) is 2. Moreover, the maximum number of pairwise disjoint, $\Pi$-noncontractible, pairwise $\Pi$-nonhomotopic cycles in $G$ is 3, by Proposition \ref{prop:disjointCycles}. 

Let $\mathcal{C}$ be a maximal peripheral family of $\Pi$-unstable cycles such that $\mathcal{C}$ is of maximum size. Then, we have two cases for the size of $\mathcal{C}$. 

\textbf{Case 1:} $|\mathcal{C}| = 1$. Let $F_1 \in \mathcal{C}$. Then, the face $F_1$ must have at least two $\Pi$-unstable edges $e_1$ and $e_2$, with each of these edges contained in exactly one other peripheral cycle. Let these peripheral cycles be $F_2$ and $F_3$, respectively. Note that $F_2$ and $F_3$ are $\Pi$-unstable and cannot form a peripheral family of cycles. Since $F_2\cup F_3$ cannot disconnect the graph, by Lemma \ref{lem:G-c1c2 conn}, these two cycles must either share an edge or be adjacent (i.e.~an edge from a vertex of $F_2$ to a vertex of $F_3$). 

Suppose first that $F_2$ and $F_3$ do not share an edge, so they are adjacent in $G$. First, we claim that these cycles must be adjacent via an edge of $F_1$. Otherwise, there is some other peripheral cycle $C$ that contains an edge from a vertex of $F_2$ to a vertex of $F_3$. Since $G$ is cubic, we can choose $C$ in such a way that it shares an edge with both $F_2$ and $F_3$. However, the four edges $e_1$ and $e_2$ and the edges from $C$ to $F_2$ and $F_3$ form an edge-cut of size four. Since $G$ is C5EC, the only 4-cuts are the trivial 4-edge cuts and since $F_2$ and $F_3$ do not share an edge, we have a contradiction. Thus, $F_2$ and $F_3$ are adjacent via an edge of $F_1$.

Now, $F_2$ and $F_3$ must each have another unstable edge. We describe the other possible unstable faces in this embedding. Let $F_4$ be the second face that shares an edge with both $F_2$ and $F_3$. Let $F_5$ and $F_6$ be the faces that share an edge with $F_1$ and with $F_2$ and $F_3$, respectively. Note that $F_5$ is only adjacent to $F_3$ if $F_1$ is a 5-cycle. Similarly, $F_1$ must be a 5-cycle for $F_6$ to be adjacent to $F_2$. Let $F_7$ and $F_8$ be the faces that share an edge with $F_4$ and with $F_2$ and $F_3$, respectively. As above, we note that $F_4$ must be a 5-cycle for $F_7$ to be adjacent to $F_3$ and $F_8$ adjacent to $F_2$. The cycles $F_5$ and $F_8$ are non-adjacent, as are $F_6$ and $F_7$, so at most one cycle from each of these pairs can be unstable. Finally, if $F_1$, $F_5$ and $F_6$ are all 5-cycles, let $F_9$ be the cycle that shares an edge with both $F_5$ and $F_6$. Note that if $F_9$ is unstable, then $F_4$, $F_7$ and $F_8$ must all be stable. Any other peripheral cycle in $G$ is disjoint from and nonadjacent to one of $F_1$, $F_2$ or $F_3$ and thus must be $\Pi$-stable. See Figure \ref{fig:g2-proof1}. 

 \begin{figure}[ht]
    \centering
    \begin{subfigure}[b]{0.35\textwidth}
         \centering
         \includegraphics[width=0.8\textwidth]{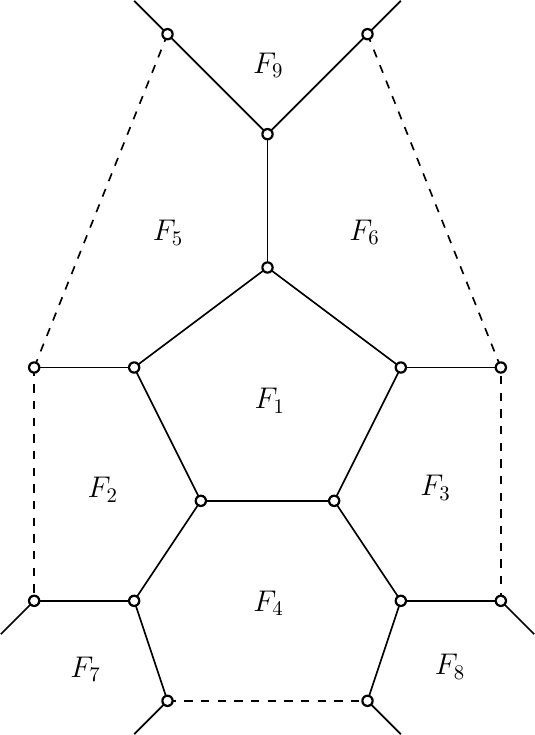}
         \caption{Faces $F_1,\dots,F_9$}
         \label{fig:g2-proof1}
     \end{subfigure}
\hspace{0.7cm}     \begin{subfigure}[b]{0.35\textwidth}
         \centering
         \includegraphics[width=0.8\textwidth]{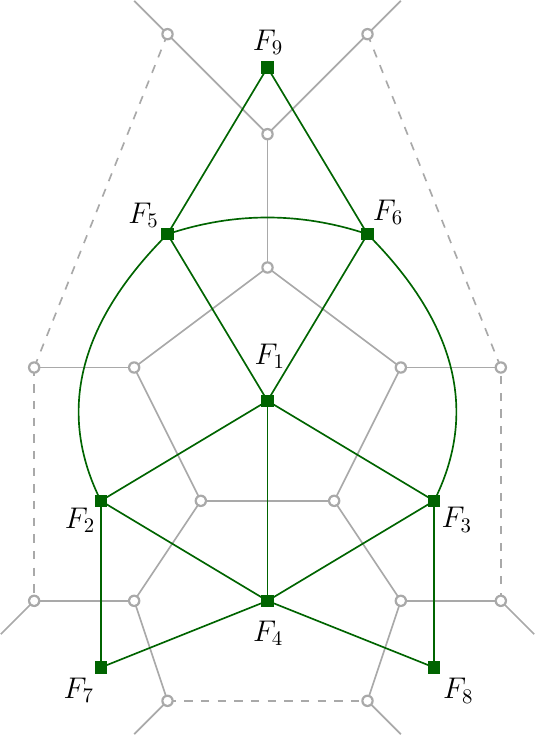}
         \caption{The dual graph}
         \label{fig:g2-proof2}
     \end{subfigure}

    \caption{The possible unstable faces in the proof of Theorem \ref{thm:g2duals}, in the case that $F_2$ and $F_3$ do not share an edge. The dual graph is indicated in green with square vertices. }
    \label{fig:g2dualsproof1}
\end{figure}

Now, we can consider the possible unstable duals by considering the subgraph of the dual graph $G^*$ induced by the vertices corresponding to these nine faces. By Proposition \ref{prop:eulerian}, the unstable dual must be Eulerian. Since the edges shared between faces $F_1$ and $F_2$ and between $F_1$ and $F_3$ are unstable, they must be included in the unstable dual. Then, to satisfy the conditions discussed above, the following pairs of vertices cannot appear together in the unstable dual: $F_4$ and $F_9$, $F_5$ and $F_8$, $F_6$ and $F_7$, $F_7$ and $F_9$, and $F_8$ and $F_9$. By generating all of the Eulerian subgraphs of this subgraph of the dual graph, shown in Figure \ref{fig:g2-proof2}, that satisfy these conditions, we obtain only unstable duals shown in Figure \ref{fig:g2duals}, as well as a 4-cycle containing the vertices $F_1$, $F_2$, $F_3$ and $F_4$. However, the 4-cycle is the unstable dual of an embedding of genus 1 and can therefore be omitted. 

Now, suppose that $F_2$ and $F_3$ do share an edge. This means that $F_1,F_2$ and $F_3$ have a common vertex, and the unstable edges in $F_1$ are incident with this vertex. Now, $F_2$ and $F_3$ must each have another unstable edge. Their shared edge may be stable or unstable but note that if it is unstable, there must be additional unstable edges, as exactly three unstable edges all incident with a single vertex gives an embedding of genus 1. As above, we describe the other possible unstable faces in an embedding of genus 2. Let $F_4$ be the face that shares a an edge with both $F_2$ and $F_3$. Note that $F_4$ is adjacent to $F_1$ via the edge shared between $F_2$ and $F_3$. Let $F_5$ and $F_6$ be the faces that share an edge with $F_1$ and $F_2$ and with $F_1$ and $F_3$, respectively. Note that $F_4$ and $F_5$ are only adjacent in $G$ if $F_2$ has length 5 and, similarly, $F_4$ and $F_6$ are only adjacent if $F_3$ has length 5. Furthermore, $F_5$ and $F_6$ are only adjacent if $F_1$ has length 5. If $F_2$ has length 5, then the face $F_7$ that shares an edge with $F_2$ and $F_4$ is adjacent to $F_1$. Similarly, if $F_3$ has length 5, then the face $F_8$ that shares an edge with $F_3$ and $F_4$ is adjacent to $F_1$. The faces $F_7$ and $F_8$ are only adjacent if $F_4$ has length 5. Finally, if $F_1$ has length 5, then let $F_9$ be the face that shares an edge with $F_5$, $F_6$ and $F_1$. See Figure \ref{fig:g2dualsproof2}. The cycles $F_5$ and $F_8$ are nonadjacent, as are $F_6$ and $F_7$, so at most one from each pair can be unstable. Finally, note that if $F_9$ is unstable, each of $F_4$, $F_7$ and $F_8$ must be stable, since they are all nonadjacent to $F_9$. 

 \begin{figure}[htb]
    \centering
    \begin{subfigure}[b]{0.35\textwidth}
         \centering
         \includegraphics[width=0.8\textwidth]{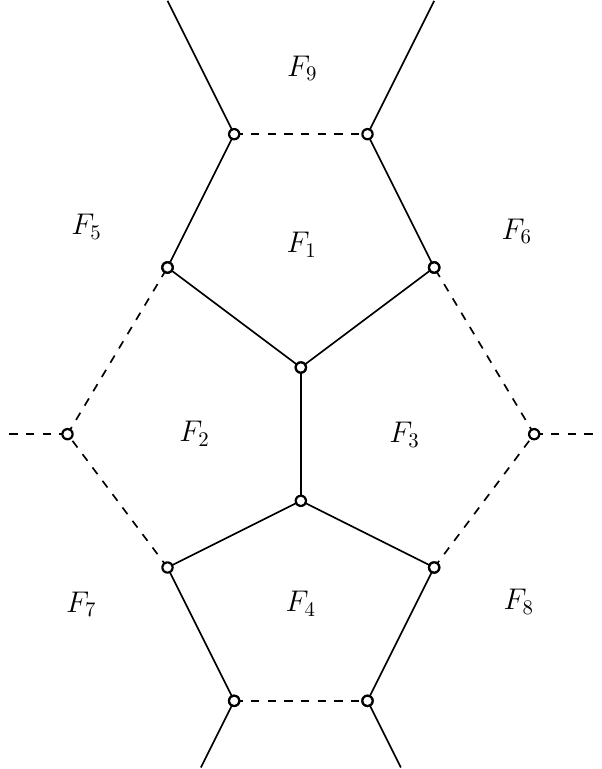}
         \caption{Faces $F_1,\dots,F_9$}
         \label{fig:g2-proof2a}
     \end{subfigure}
\hspace{0.7cm}     \begin{subfigure}[b]{0.35\textwidth}
         \centering
         \includegraphics[width=0.8\textwidth]{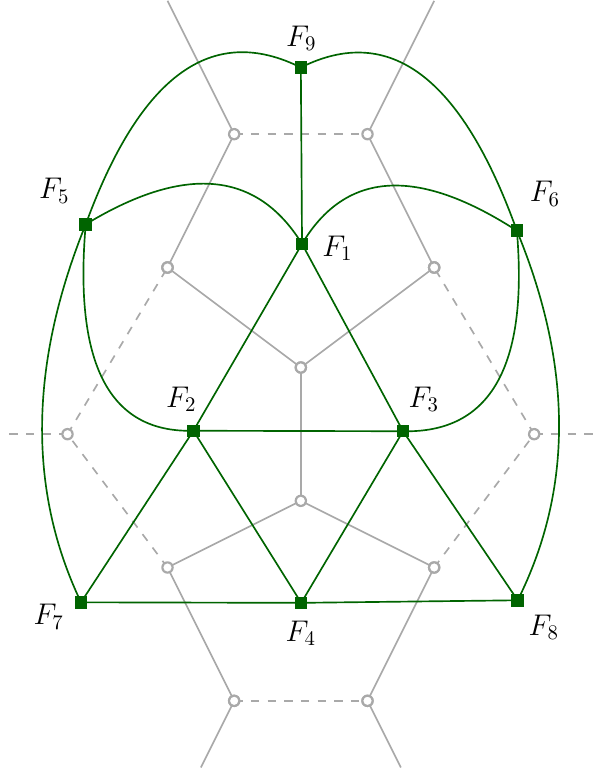}
         \caption{The dual graph}
         \label{fig:g2-proof2b}
     \end{subfigure}

    \caption{The possible unstable faces in the proof of Theorem \ref{thm:g2duals}, in the case that $F_2$ and $F_3$ do not share an edge. The dual graph is indicated in green with square vertices. }
    \label{fig:g2dualsproof2}
\end{figure}

Now, we can consider possible unstable duals by considering the subgraph of the dual graph $G^*$ induced by the vertices corresponding to these nine faces. The edges shared between faces $F_1$ and $F_2$ and between $F_1$ and $F_3$ are unstable, so they must be included in the unstable dual. Then, to satisfy the conditions above, the following pairs of vertices cannot appear together in the unstable dual: $F_4$ and $F_9$, $F_5$ and $F_8$, $F_6$ and $F_7$, $F_7$ and $F_9$, and $F_8$ and $F_9$. By generating all of the Eulerian subgraphs of the described subgraph of the dual that satisfy these conditions, we obtain only unstable duals shown in Figure \ref{fig:g2duals}, as well as a 3-cycle containing the vertices $F_1$, $F_2$ and $F_3$ and a 4-cycle containing the vertices $F_1$, $F_2$, $F_3$ and $F_4$. However, the 3-cycle and 4-cycle are unstable duals of embeddings of genus 1 and can therefore be omitted. 

\textbf{Case 2:} $|\mathcal{C}| = 2$. Let $F_1,F_2\in \mathcal{C}$. By Proposition \ref{prop:disconnected dual}, if the $\Pi$-unstable dual is disconnected, each component must be a 3-cycle or a 4-cycle, since these are the two possible unstable duals of embeddings with genus 1. Thus, we can consider only the cases in which the $\Pi$-unstable dual is connected. 

Suppose first that there is some peripheral cycle $F_3$ such that $F_1$, $F_2$ and $F_3$ are pairwise disjoint, $\Pi$-noncontractible and pairwise $\Pi$-nonhomotopic. Then, since $g(\Pi)=2$, it must be the case that $F_3$ is adjacent to one of $F_1$ or $F_2$, or $G-V(F_1\cup F_2\cup F_3)$ is disconnected. Assume the latter, and let $A$ be a component of $G-V(F_1\cup F_2\cup F_3)$. Since no pair of $F_1$, $F_2$ and $F_3$ disconnects $G$, each of these three cycles must have at least one edge to $A$. 

We now argue that at least two of the cycles $F_1$, $F_2$ and $F_3$ must have at least two edges to $A$, unless $A$ contains a single vertex. 

Suppose first that $A$ is acyclic and contains more than one vertex. Then, the subgraph of $G$ induced by the vertices of $A$ contains at least two leaves and, because $G$ is cubic, there must be at least two edges from at least one of $F_1$, $F_2$ and $F_3$ to $A$. Without loss of generality, suppose there are at least two edges from $F_1$. If both $F_2$ and $F_3$ have only a single edge to $A$, we can identify a 4-cut that separates two cycles in $G$. As shown in the proof of Lemma \ref{lem:one edge}, the vertices incident with an edge from $F_1$ to $A$ are all consecutive on the cycle $F_1$, forming the $A$-interval on $F_1$. There is a cycle in $G$ formed by two of the edges from $F_1$ to $A$, the path between the endpoints of these edges in $F_1$ contained in the $A$-interval, and a path between the endpoints of these edges in $A$. Now, deleting the edges in $F_1$ that have exactly one endpoint in the $A$-interval, the edge from $F_2$ to $A$, and the edge from $F_3$ to $A$ separates the cycles $F_2$ and $F_3$ from the cycle described above, giving a contradiction. Thus, at least two of $F_1$, $F_2$ and $F_3$ must have two edges to $A$. 

Similarly, if $A$ contains a cycle then because $G$ is C5EC, there must be at least five edges from $F_1\cup F_2\cup F_3$ to $A$, with each of the three cycles having at least one edge to $A$. An argument similar to that of the acyclic case shows that at least two of the three cycles have at least two edges to $A$. 

Next, we claim that there can only be two components in the graph $G-V(F_1\cup F_2 \cup F_3)$. Suppose otherwise, and let $A_1$, $A_2$, and $A_3$ be components of $G-V(F_1\cup F_2 \cup F_3)$. Construct the graph $\widehat{G}$ from $G$ by adding a vertex $u_i$ in the interior of the face $F_i$, for $i=,1,2,3$, and adding an edge from this new vertex to every vertex in the corresponding $F_i$. Clearly, since $G$ is planar, $\widehat{G}$ is also planar. 

However, $\widehat{G}$ contains a $K_{3,3}$ minor. To show this, we show that each component $A_i$ contains a vertex $v_i$ with disjoint paths to each of the cycles $F_1$, $F_2$, $F_3$. Take any vertex $v_i\in V(A_i)$ and consider paths $Q_1$, $Q_2$ and $Q_3$ in $\widehat{G}$ from $v_i$ to $u_1$, $u_2$ and $u_3$, respectively, such that $Q_1$ does not contain any vertices in $F_2$ or $F_3$, and similarly for the paths $Q_2$ and $Q_3$. If these paths are disjoint, then we are done. Otherwise, if $Q_1$ and $Q_2$ intersect, consider the intersection closest to $v_i$ on the paths. Since $G$ is cubic, every vertex in $A_i$ in $\widehat{G}$ has degree 3, so $Q_1$ and $Q_2$ must intersect at an edge. Let $x$ be the endpoint of this shared edge that is closer to $v_i$ on the path $Q_1$, and $y$ the other endpoint of the shared edge. If $x$ also appears first on the path from $v_i$ to $u_2$, then replace the vertex $v_i$ under consideration by $y$, and let $Q'_1$ be the segment of $Q_1$ from $y$ to $u_1$, $Q'_2$ the segment of $Q_2$ from $y$ to $u_2$, and $Q'_3$ the path $Q_3$ appended to the segment of $Q_1$ from $v_i$ to $y$. Otherwise, if $y$ appears first on the path $Q_2$ from $v_i$ to $u_2$, replace the vertex under consideration with $x$ and construct $Q'_1$, $Q'_2$ and $Q'_3$ as above. These paths have at least one intersection less than $Q_1$, $Q_2$ and $Q_3$. Repeat this process until there are no intersections on the paths to $u_1$ and $u_2$. Then repeat with the paths to $u_2$ and $u_3$ to obtain a set of three paths without intersections. Then, through this process, we obtain three vertices $v_1$, $v_2$ and $v_3$ and paths to $u_1$, $u_2$ and $u_3$ that form a $K_{3,3}$ minor in $\widehat{G}$. This contradicts $G$ being planar and therefore there can be only two components in $G-V(F_1\cup F_2\cup F_3)$. 

\begin{figure}[t]
    \centering
    \begin{subfigure}[b]{0.35\textwidth}
         \centering
         \includegraphics[width=\textwidth]{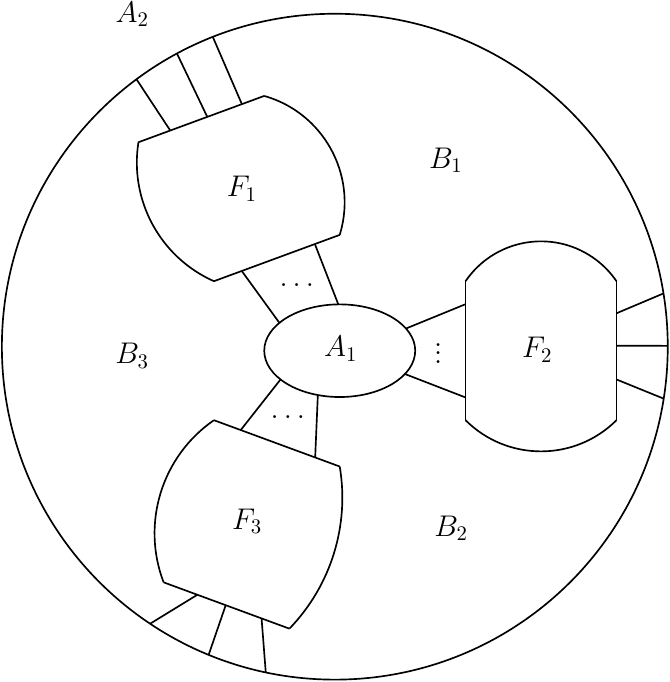}
         \caption{The structure of the three separating faces in $G$}
         \label{fig:g2-proofcase3sep}
     \end{subfigure}
\hspace{1cm}     \begin{subfigure}[b]{0.35\textwidth}
         \centering
         \includegraphics[width=\textwidth]{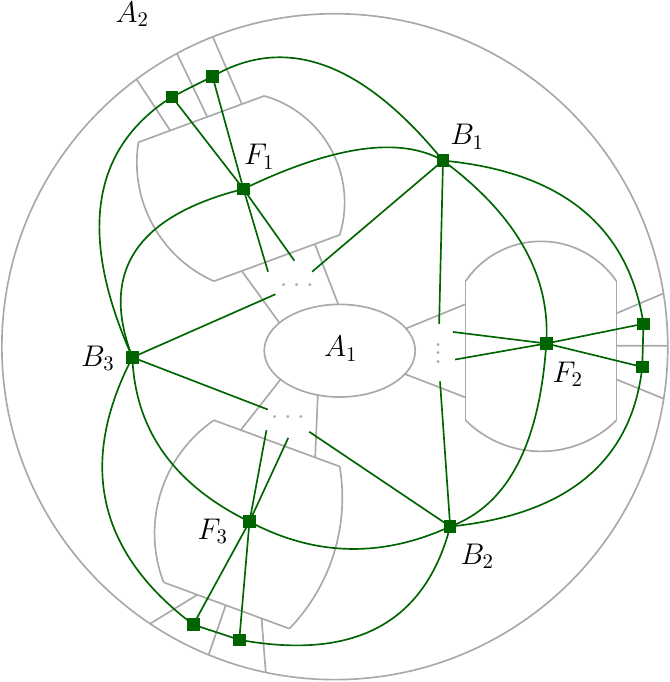}
         \caption{The dual graph}
         \label{fig:g2-proofcase3sepb}
     \end{subfigure}

    \caption{The $\Pi$-unstable faces $F_1$, $F_2$ and $F_3$ and their neighboring faces, with the corresponding vertices in the dual. The dual graph is indicated in green with square vertices. Note that one of the three cycles may have only one edge to $A_1$ or to $A_2$.}
    \label{fig:g2dualsproofcase2sep}
\end{figure}

Now, with only two components in $G-V(F_1\cup F_2\cup F_3)$ and the edges from each $F_i$ to each of these components appearing consecutively on $F_i$, there must exist faces $B_1$, $B_2$ and $B_3$ such that $F_1,B_1,F_2,B_2,F_3,B_3$ forms a 6-cycle in the dual graph $G^*$. The edge shared between cycles $F_1$ and $B_1$ is an edge that appears between the $A_1$-interval and the $A_2$-interval on $F_1$, where $A_1$ and $A_2$ are the two components of $G-V(F_1\cup F_2\cup F_3)$. Similarly for the other pairs of faces. See Figure \ref{fig:g2-proofcase3sep}. 

Each unstable face $F_i$ has at least two unstable edges, so there must be additional $\Pi$-unstable faces. If there is a pair of these additional unstable faces that are both disjoint from $F_2$ and $F_3$, then these faces must share an edge with each other. Similarly for unstable faces disjoint from $F_1$ and $F_2$ or from $F_1$ and $F_3$. 

Figure \ref{fig:g2-proofcase3sepb} shows the possible vertices of the $\Pi$-unstable dual. From above, we know that the vertices corresponding to $F_1$, $F_2$ and $F_3$ must each have degree at least two and that any pair of vertices in the unstable dual that are both adjacent to $F_1$ (similarly, to $F_2$ or $F_3$) and not to $F_2$ and $F_3$ must themselves be adjacent vertices. Further note that in order for the unstable dual to be connected, at least one of $B_1$, $B_2$ and $B_3$ must be $\Pi$-unstable. 

As above, we can examine all of the Eulerian subgraphs of the dual containing vertices that satisfy these conditions. Those that correspond to embeddings of genus 2 can all be found in Figure \ref{fig:g2duals}.

Now, suppose there is no third cycle disjoint from and non-adjacent to $F_1$ and $F_2$ whose union with $F_1$ and $F_2$ disconnects the graph. Then, any other $\Pi$-unstable peripheral cycle must either share an edge with or be adjacent to one of $F_1$ or $F_2$. Both $F_1$ and $F_2$ must be non-zero-homologous and thus we can cut along $F_1$ and $F_2$ to obtain a planar embedding. Note that this operation only affects those vertices in $N[V(F_1\cup F_2)]$. A planar embedding of a cubic C5EC planar graph has no unstable vertices. Thus, the only possible $\Pi$-unstable vertices are those in $N[V(F_1\cup F_2)]$. 

 \begin{figure}[hbt]
    \centering
    \begin{subfigure}[b]{0.35\textwidth}
         \centering
         \includegraphics[width=\textwidth]{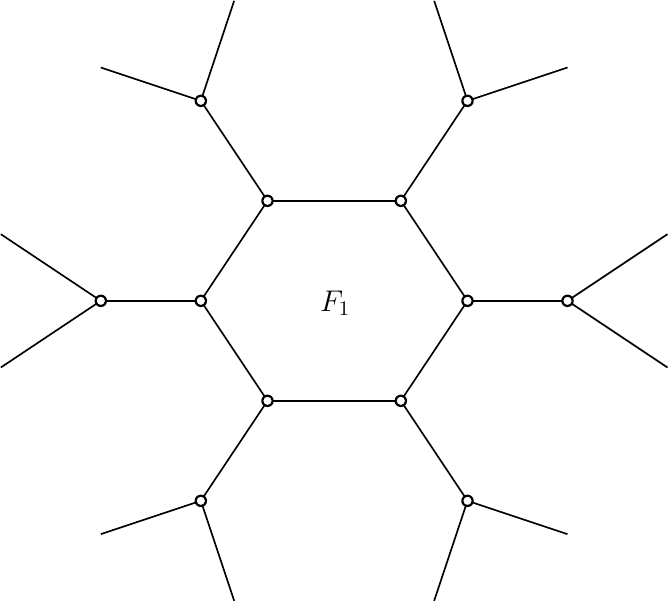}
         \caption{The possible $\Pi$-unstable edges with a $\Pi$-unstable endpoint in the neighborhood of $F_1$}
         \label{fig:g2-proofcase2}
     \end{subfigure}
\hspace{1cm}     \begin{subfigure}[b]{0.35\textwidth}
         \centering
         \includegraphics[width=\textwidth]{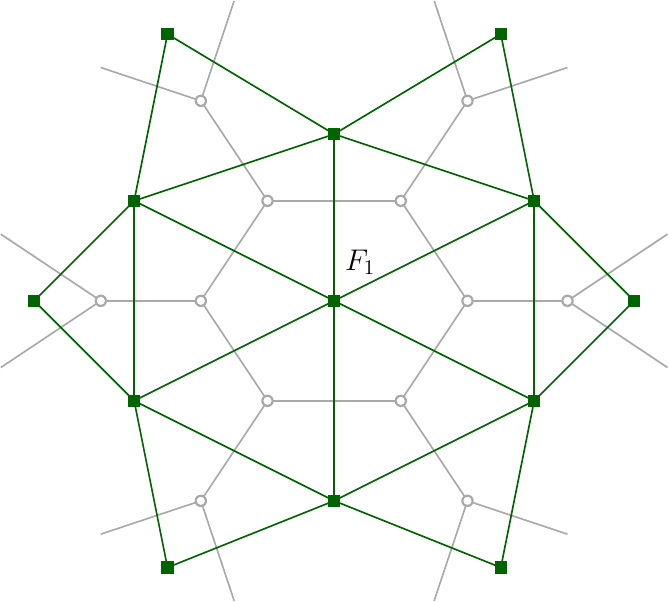}
         \caption{The dual graph}
         \label{fig:g2-proofcase2b}
     \end{subfigure}

    \caption{The possible $\Pi$-unstable vertices in the neighborhood of $F_1$ and the corresponding vertices of the dual. The dual graph is indicated in green with square vertices. }
    \label{fig:g2dualsproofcase2}
\end{figure}

In this way, we can determine possible $\Pi$-unstable edges on faces that share an edge with or are adjacent to $F_1$, and therefore also the possible edges of the $\Pi$-unstable dual. See Figure \ref{fig:g2dualsproofcase2}. We can similarly determine the possible $\Pi$-unstable faces in the neighborhood of $F_2$. Note that because both $F_1$ and $F_2$ are $\Pi$-unstable, they must both contain at least two $\Pi$-unstable edges and have degree at least two in the dual. Moreover, $F_1$ and $F_2$ are disjoint and non-adjacent, so none of the vertices of the dual in Figure \ref{fig:g2-proofcase2b} correspond to the face $F_2$. 

Using these facts, we can determine all of the possible unstable duals containing no three vertices corresponding to a peripheral family of cycles, or to three cycles whose union disconnects the graph. Those that correspond to embeddings of genus 2 can all be found in Figure \ref{fig:g2duals}. 
\end{proof}

\section{Conclusion and future work}
\label{section:future}

In this paper, we have explored new methods of describing the 2-cell embeddings of cubic graphs. Using the unstable dual, calculating the genus distribution of a cubic graph can be transformed into a subgraph counting problem. 

In Section~\ref{section:g2}, we characterize the unstable duals of a cubic C5EC planar graph corresponding to embeddings with genus 2. 
Using these graphs, we can construct unstable duals corresponding to an embedding with genus 3 of a cubic C5EC planar graph, using the results of Section \ref{section:cuts and duals}. Those that are 2-connected are shown in Figure \ref{fig:g3duals}. However, the number of such graphs makes it difficult to obtain a precise characterization of all unstable duals of embeddings with genus 3, as in the genus 2 case above. A full characterization of the unstable duals corresponding to embeddings of genus $k\geq 3$ in a cubic C5EC planar graph remains open. 

  \begin{figure}[!b]
    \centering
    \hfill
    \begin{subfigure}[b]{0.09\textwidth}
         \centering
         \includegraphics[width=\textwidth]{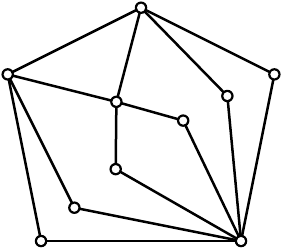}%
     \end{subfigure}
\hfill
\begin{subfigure}[b]{0.09\textwidth}
         \centering
         \includegraphics[width=\textwidth]{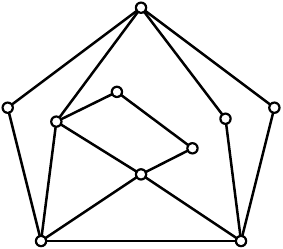}%
     \end{subfigure}
     \hfill
     \begin{subfigure}[b]{0.09\textwidth}
         \centering
         \includegraphics[width=\textwidth]{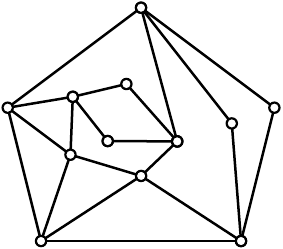}%
     \end{subfigure}
     \hfill
    \begin{subfigure}[b]{0.09\textwidth}
         \centering
         \includegraphics[width=\textwidth]{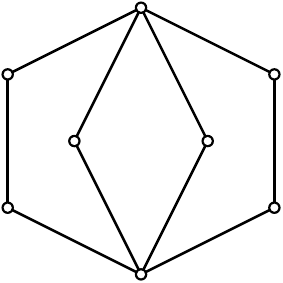}%
     \end{subfigure}
     \hfill
     \begin{subfigure}[b]{0.09\textwidth}
         \centering
         \includegraphics[width=\textwidth]{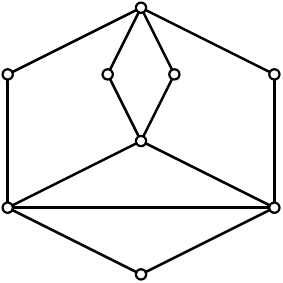}
     \end{subfigure}
\hfill
     \begin{subfigure}[b]{0.09\textwidth}
         \centering
         \includegraphics[width=\textwidth]{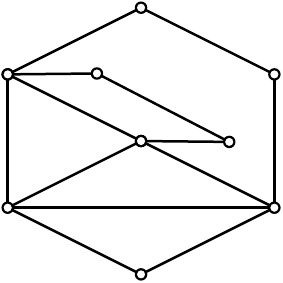}
     \end{subfigure}
     \hfill
     \begin{subfigure}[b]{0.09\textwidth}
         \centering
         \includegraphics[width=\textwidth]{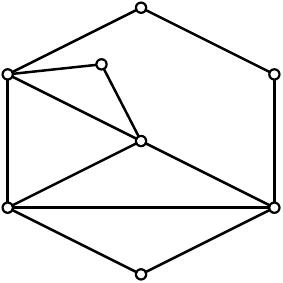}
     \end{subfigure}
\hfill
     \begin{subfigure}[b]{0.09\textwidth}
         \centering
         \includegraphics[width=\textwidth]{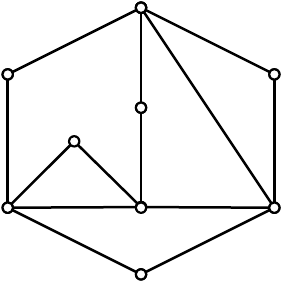}
     \end{subfigure}
\hfill
     \begin{subfigure}[b]{0.09\textwidth}
         \centering
         \includegraphics[width=\textwidth]{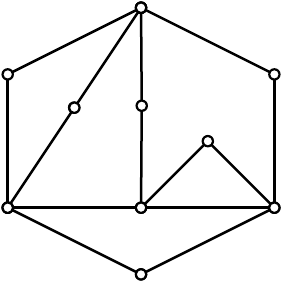} 
     \end{subfigure}
     \hfill
     \hfill
     
     \vspace{0.3cm}

\hfill
     \begin{subfigure}[b]{0.09\textwidth}
         \centering
         \includegraphics[width=\textwidth]{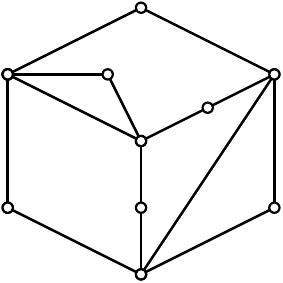}
     \end{subfigure}
\hfill
     \begin{subfigure}[b]{0.09\textwidth}
         \centering
         \includegraphics[width=\textwidth]{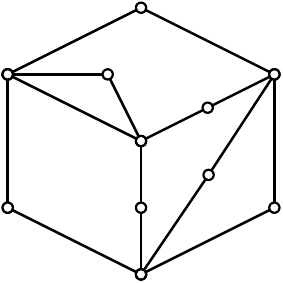}
     \end{subfigure}
\hfill
     \begin{subfigure}[b]{0.09\textwidth}
         \centering
         \includegraphics[width=\textwidth]{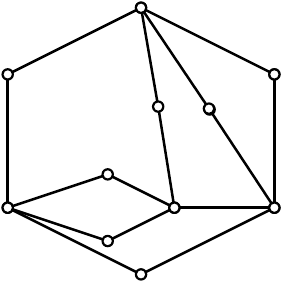}
     \end{subfigure}
     \hfill
     \begin{subfigure}[b]{0.09\textwidth}
         \centering
         \includegraphics[width=\textwidth]{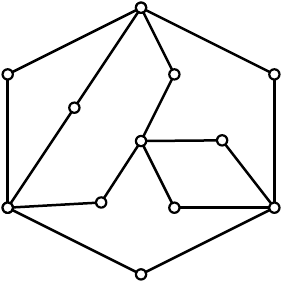}
     \end{subfigure}
\hfill
     \begin{subfigure}[b]{0.09\textwidth}
         \centering
         \includegraphics[width=\textwidth]{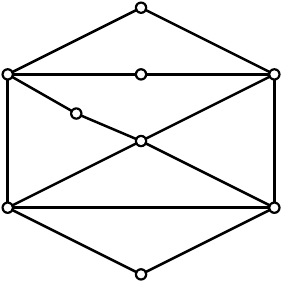}
     \end{subfigure}
\hfill
     \begin{subfigure}[b]{0.09\textwidth}
         \centering
         \includegraphics[width=\textwidth]{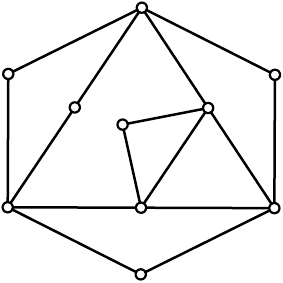}
     \end{subfigure}
     \hfill
     \begin{subfigure}[b]{0.09\textwidth}
         \centering
         \includegraphics[width=\textwidth]{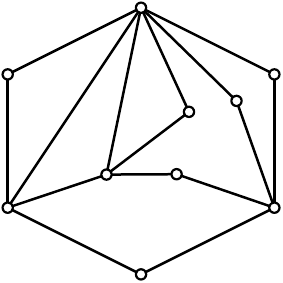}
     \end{subfigure}
\hfill
     \begin{subfigure}[b]{0.09\textwidth}
         \centering
         \includegraphics[width=\textwidth]{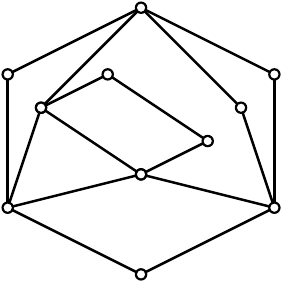}
     \end{subfigure}
\hfill
     \begin{subfigure}[b]{0.09\textwidth}
         \centering
         \includegraphics[width=\textwidth]{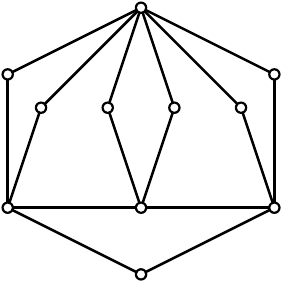}
     \end{subfigure}
     \hfill
     \hfill
     
     \vspace{0.3cm}

\hfill
     \begin{subfigure}[b]{0.09\textwidth}
         \centering
         \includegraphics[width=\textwidth]{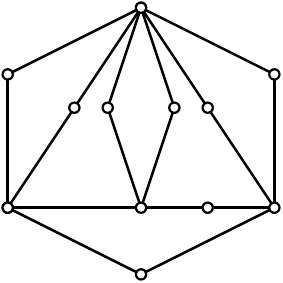}
     \end{subfigure}
\hfill
     \begin{subfigure}[b]{0.09\textwidth}
         \centering
         \includegraphics[width=\textwidth]{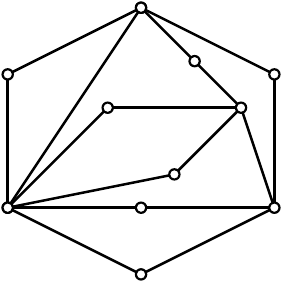}
     \end{subfigure}
     \hfill
     \begin{subfigure}[b]{0.09\textwidth}
         \centering
         \includegraphics[width=\textwidth]{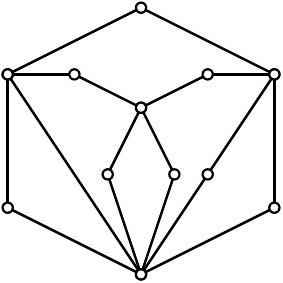}
     \end{subfigure}
     \hfill
     \begin{subfigure}[b]{0.09\textwidth}
         \centering
         \includegraphics[width=\textwidth]{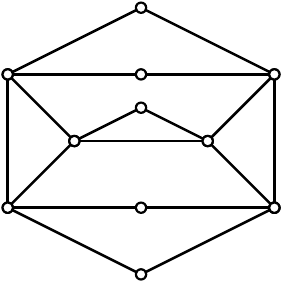}
     \end{subfigure}
\hfill
     \begin{subfigure}[b]{0.09\textwidth}
         \centering
         \includegraphics[width=\textwidth]{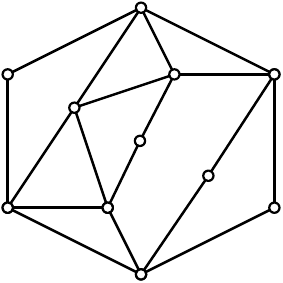}
     \end{subfigure}
\hfill
     \begin{subfigure}[b]{0.09\textwidth}
         \centering
         \includegraphics[width=\textwidth]{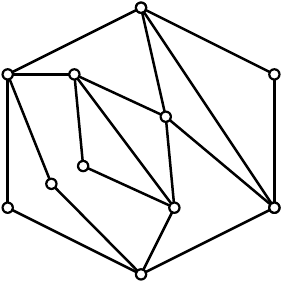}
     \end{subfigure}
     \hfill
     \begin{subfigure}[b]{0.09\textwidth}
         \centering
         \includegraphics[width=\textwidth]{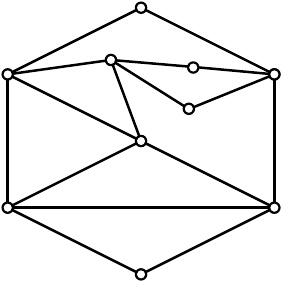}
     \end{subfigure}
\hfill
     \begin{subfigure}[b]{0.09\textwidth}
         \centering
         \includegraphics[width=\textwidth]{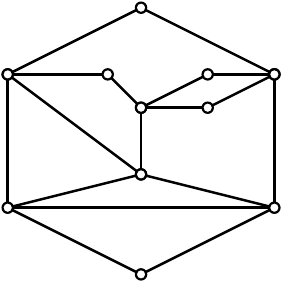}
     \end{subfigure}
\hfill
     \begin{subfigure}[b]{0.09\textwidth}
         \centering
         \includegraphics[width=\textwidth]{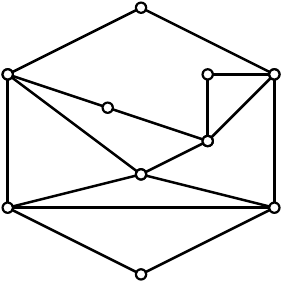}
     \end{subfigure}
     \hfill
     \hfill
     
\vspace{0.3cm}
\hfill
     \begin{subfigure}[b]{0.09\textwidth}
         \centering
         \includegraphics[width=\textwidth]{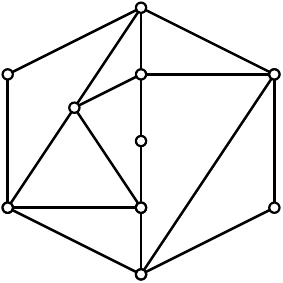}
     \end{subfigure}
     \hfill
     \begin{subfigure}[b]{0.09\textwidth}
         \centering
         \includegraphics[width=\textwidth]{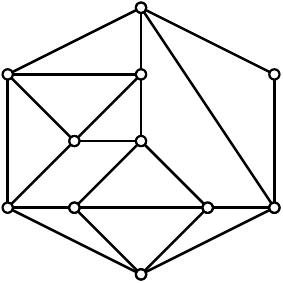}
     \end{subfigure}
\hfill
     \begin{subfigure}[b]{0.09\textwidth}
         \centering
         \includegraphics[width=\textwidth]{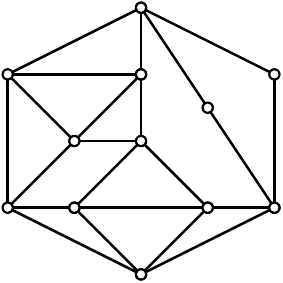}
     \end{subfigure}
     \hfill
     \begin{subfigure}[b]{0.09\textwidth}
         \centering
         \includegraphics[width=\textwidth]{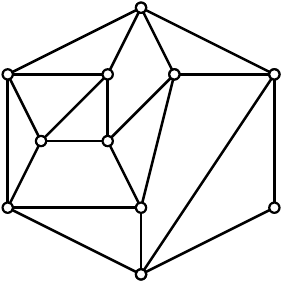}
     \end{subfigure}
\hfill
     \begin{subfigure}[b]{0.09\textwidth}
         \centering
         \includegraphics[width=\textwidth]{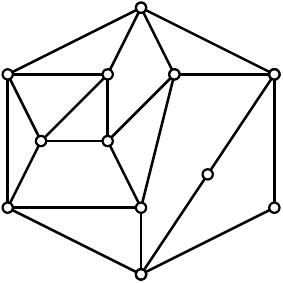}
     \end{subfigure}
\hfill
     \begin{subfigure}[b]{0.09\textwidth}
         \centering
         \includegraphics[width=\textwidth]{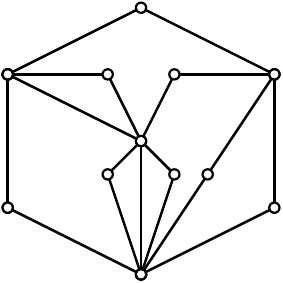}
     \end{subfigure}
     \hfill
     \begin{subfigure}[b]{0.09\textwidth}
         \centering
         \includegraphics[width=\textwidth]{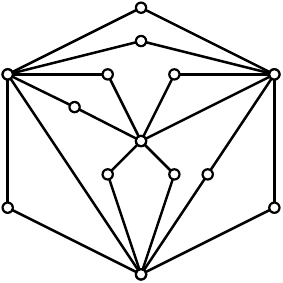}
     \end{subfigure}
\hfill
     \begin{subfigure}[b]{0.09\textwidth}
         \centering
         \includegraphics[width=\textwidth]{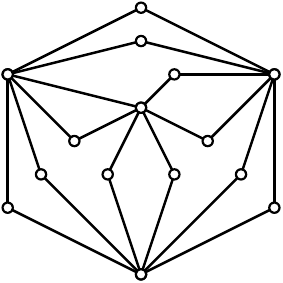}
     \end{subfigure}
\hfill
     \begin{subfigure}[b]{0.09\textwidth}
         \centering
         \includegraphics[width=\textwidth]{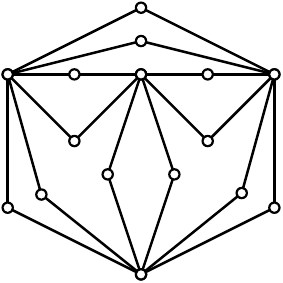}
     \end{subfigure}
     \hfill
     \hfill
     
\end{figure}

\newpage

\begin{figure}[!ht]
\centering
\ContinuedFloat
\hfill
     \begin{subfigure}[b]{0.09\textwidth}
         \centering
         \includegraphics[width=\textwidth]{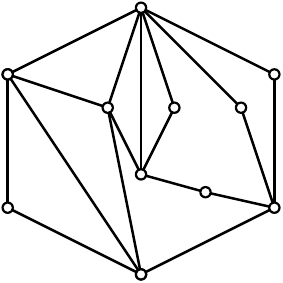}
     \end{subfigure}
\hfill
     \begin{subfigure}[b]{0.09\textwidth}
         \centering
         \includegraphics[width=\textwidth]{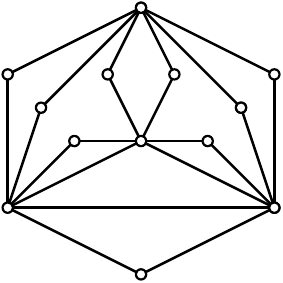}
     \end{subfigure}
\hfill
     \begin{subfigure}[b]{0.09\textwidth}
         \centering
         \includegraphics[width=\textwidth]{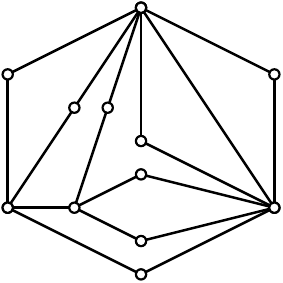}
     \end{subfigure}
     \hfill
     \begin{subfigure}[b]{0.09\textwidth}
         \centering
         \includegraphics[width=\textwidth]{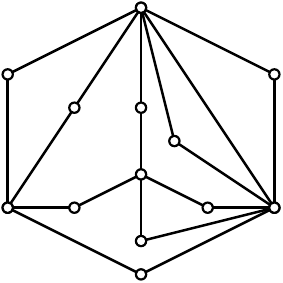}
     \end{subfigure}
\hfill
     \begin{subfigure}[b]{0.09\textwidth}
         \centering
         \includegraphics[width=\textwidth]{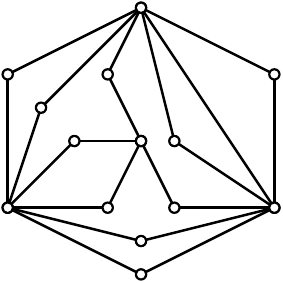}
     \end{subfigure}
\hfill
     \begin{subfigure}[b]{0.09\textwidth}
         \centering
         \includegraphics[width=\textwidth]{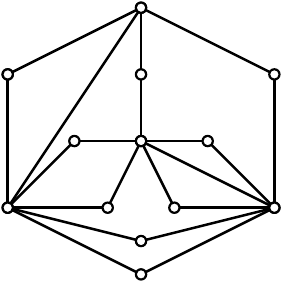}
     \end{subfigure}
     \hfill
     \begin{subfigure}[b]{0.09\textwidth}
         \centering
         \includegraphics[width=\textwidth]{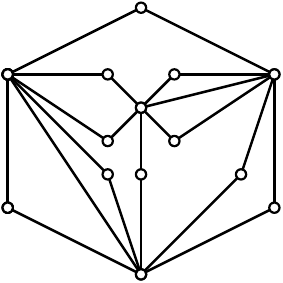}
     \end{subfigure}
\hfill
     \begin{subfigure}[b]{0.09\textwidth}
         \centering
         \includegraphics[width=\textwidth]{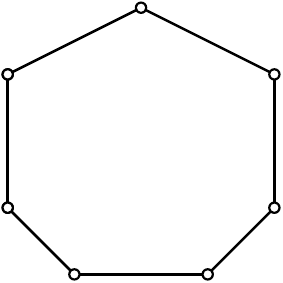}
     \end{subfigure}
     \hfill
     \begin{subfigure}[b]{0.09\textwidth}
         \centering
         \includegraphics[width=\textwidth]{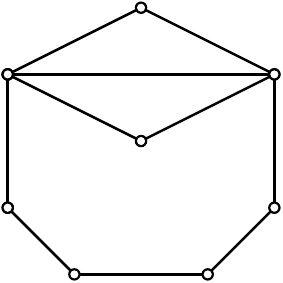}
     \end{subfigure}
\hfill
\hfill
     
     \vspace{0.3cm}

\hfill
      \begin{subfigure}[b]{0.09\textwidth}
         \centering
         \includegraphics[width=\textwidth]{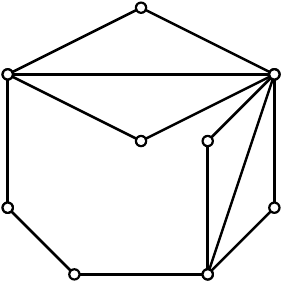}
     \end{subfigure}
     \hfill
     \begin{subfigure}[b]{0.09\textwidth}
         \centering
         \includegraphics[width=\textwidth]{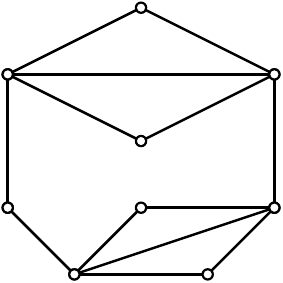}
     \end{subfigure}
\hfill
     \begin{subfigure}[b]{0.09\textwidth}
         \centering
         \includegraphics[width=\textwidth]{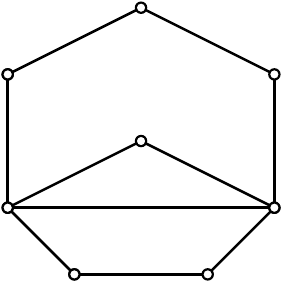}
     \end{subfigure}
          \hfill
     \begin{subfigure}[b]{0.09\textwidth}
         \centering
         \includegraphics[width=\textwidth]{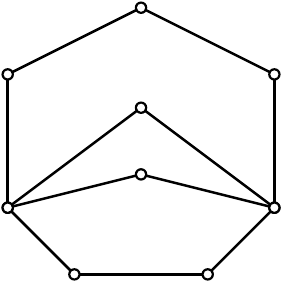}
     \end{subfigure}
\hfill
     \begin{subfigure}[b]{0.09\textwidth}
         \centering
         \includegraphics[width=\textwidth]{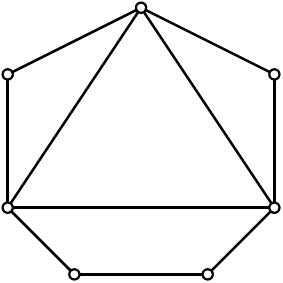}
     \end{subfigure}
     \hfill
     \begin{subfigure}[b]{0.09\textwidth}
         \centering
         \includegraphics[width=\textwidth]{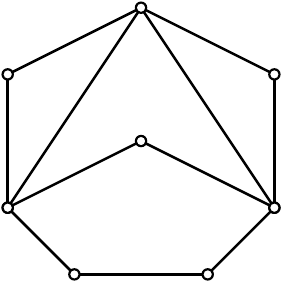}
     \end{subfigure}
\hfill
     \begin{subfigure}[b]{0.09\textwidth}
         \centering
         \includegraphics[width=\textwidth]{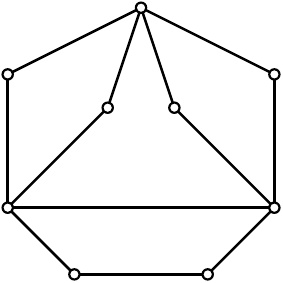}
     \end{subfigure}     \hfill
     \begin{subfigure}[b]{0.09\textwidth}
         \centering
         \includegraphics[width=\textwidth]{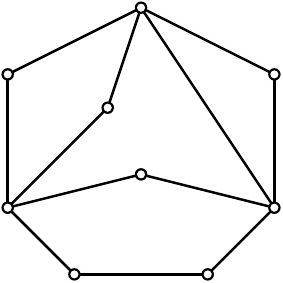}
     \end{subfigure}
\hfill
     \begin{subfigure}[b]{0.09\textwidth}
         \centering
         \includegraphics[width=\textwidth]{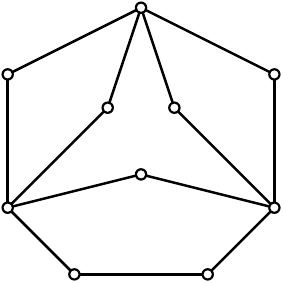}
     \end{subfigure}
\hfill
\hfill
     
\vspace{0.3cm}
\hfill
     \begin{subfigure}[b]{0.09\textwidth}
         \centering
         \includegraphics[width=\textwidth]{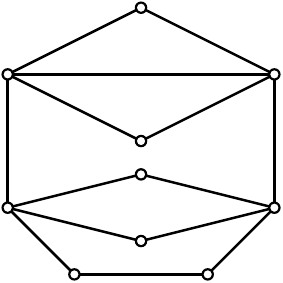}
     \end{subfigure}
          \hfill
     \begin{subfigure}[b]{0.09\textwidth}
         \centering
         \includegraphics[width=\textwidth]{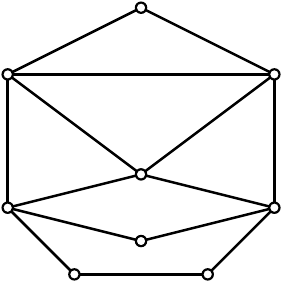}
     \end{subfigure}
\hfill
     \begin{subfigure}[b]{0.09\textwidth}
         \centering
         \includegraphics[width=\textwidth]{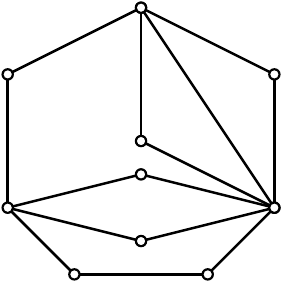}
     \end{subfigure}
\hfill
     \begin{subfigure}[b]{0.09\textwidth}
         \centering
         \includegraphics[width=\textwidth]{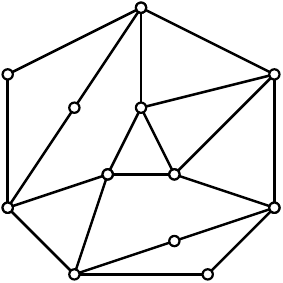}
     \end{subfigure}
\hfill
     \begin{subfigure}[b]{0.09\textwidth}
         \centering
         \includegraphics[width=\textwidth]{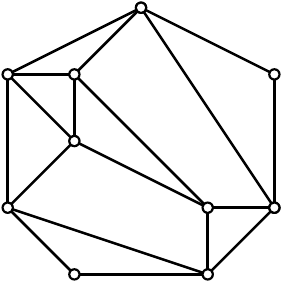}
     \end{subfigure}
\hfill
     \begin{subfigure}[b]{0.09\textwidth}
         \centering
         \includegraphics[width=\textwidth]{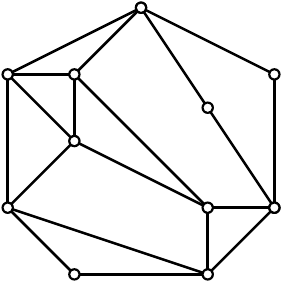}
     \end{subfigure}     \hfill
     \begin{subfigure}[b]{0.09\textwidth}
         \centering
         \includegraphics[width=\textwidth]{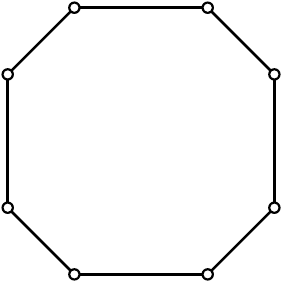}
     \end{subfigure}
\hfill
     \begin{subfigure}[b]{0.09\textwidth}
         \centering
         \includegraphics[width=\textwidth]{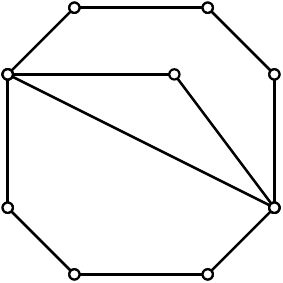}
     \end{subfigure}
\hfill
     \begin{subfigure}[b]{0.09\textwidth}
         \centering
         \includegraphics[width=\textwidth]{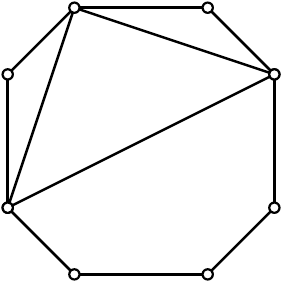}
     \end{subfigure}
      \hfill
      \hfill

      \vspace{0.3cm}

      \hfill
      \begin{subfigure}[b]{0.09\textwidth}
         \centering
         \includegraphics[width=\textwidth]{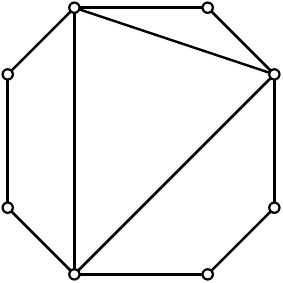}
     \end{subfigure}
     \hfill      \begin{subfigure}[b]{0.09\textwidth}
         \centering
         \includegraphics[width=\textwidth]{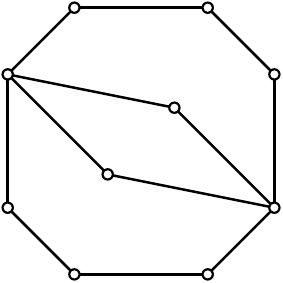}
     \end{subfigure}
      \hfill
      \begin{subfigure}[b]{0.09\textwidth}
         \centering
         \includegraphics[width=\textwidth]{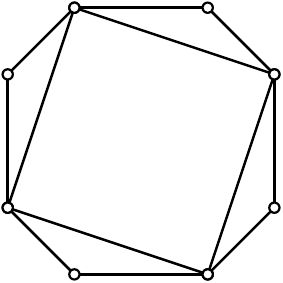}
     \end{subfigure}
     \hfill
      \begin{subfigure}[b]{0.09\textwidth}
         \centering
         \includegraphics[width=\textwidth]{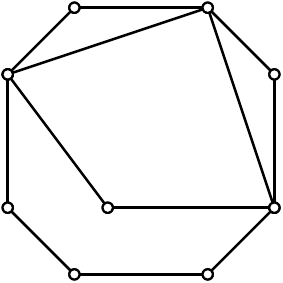}
     \end{subfigure}
     \hfill
      \begin{subfigure}[b]{0.09\textwidth}
         \centering
         \includegraphics[width=\textwidth]{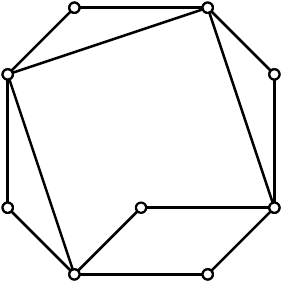}
     \end{subfigure}
      \hfill
      \begin{subfigure}[b]{0.09\textwidth}
         \centering
         \includegraphics[width=\textwidth]{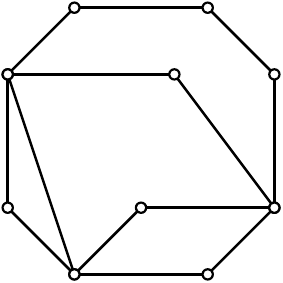}
     \end{subfigure}
      \hfill    
      \begin{subfigure}[b]{0.09\textwidth}
         \centering
         \includegraphics[width=\textwidth]{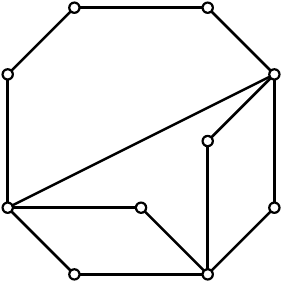}
     \end{subfigure}
      \hfill
      \begin{subfigure}[b]{0.09\textwidth}
         \centering
         \includegraphics[width=\textwidth]{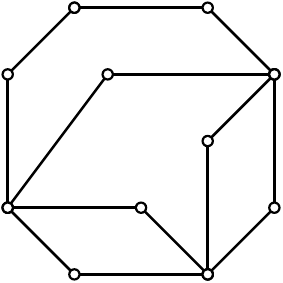}
     \end{subfigure}
      \hfill
      \begin{subfigure}[b]{0.09\textwidth}
         \centering
         \includegraphics[width=\textwidth]{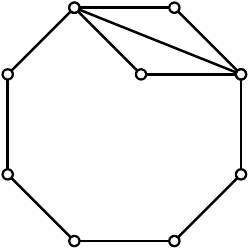}
     \end{subfigure}
     \hfill
     \hfill
     
   \vspace{0.3cm}
   
\hfill
      \begin{subfigure}[b]{0.09\textwidth}
         \centering
         \includegraphics[width=\textwidth]{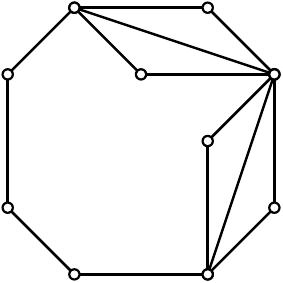}
     \end{subfigure}
     \hfill
      \begin{subfigure}[b]{0.09\textwidth}
         \centering
         \includegraphics[width=\textwidth]{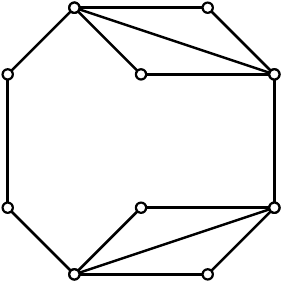}
     \end{subfigure}
      \hfill
      \begin{subfigure}[b]{0.09\textwidth}
         \centering
         \includegraphics[width=\textwidth]{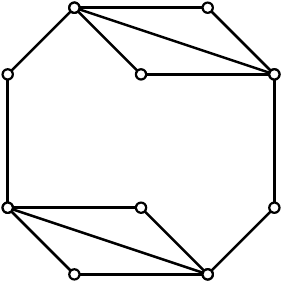}
     \end{subfigure}
     \hfill
      \begin{subfigure}[b]{0.09\textwidth}
         \centering
         \includegraphics[width=\textwidth]{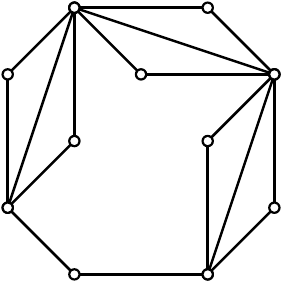}
     \end{subfigure}
      \hfill
      \begin{subfigure}[b]{0.09\textwidth}
         \centering
         \includegraphics[width=\textwidth]{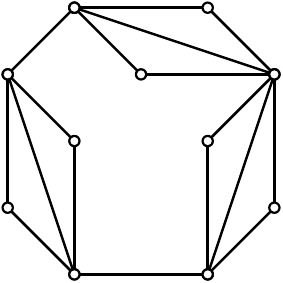}
     \end{subfigure}
       \hfill  
      \begin{subfigure}[b]{0.09\textwidth}
         \centering
         \includegraphics[width=\textwidth]{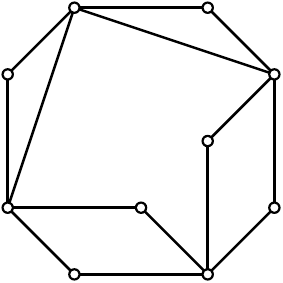}
     \end{subfigure}
      \hfill
      \begin{subfigure}[b]{0.09\textwidth}
         \centering
         \includegraphics[width=\textwidth]{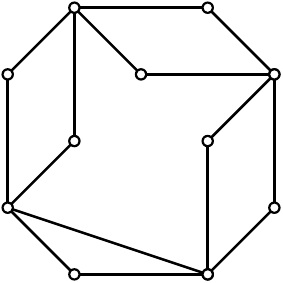}
     \end{subfigure}
      \hfill
      \begin{subfigure}[b]{0.09\textwidth}
         \centering
         \includegraphics[width=\textwidth]{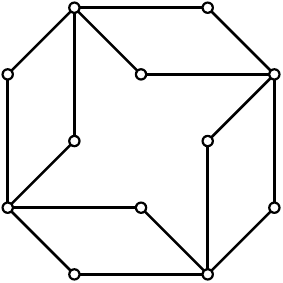}
     \end{subfigure}
     \hfill
      \begin{subfigure}[b]{0.09\textwidth}
         \centering
         \includegraphics[width=\textwidth]{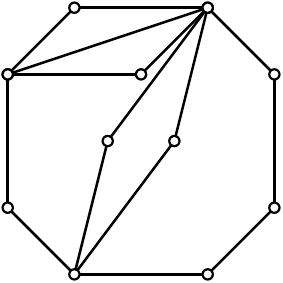}
     \end{subfigure}
      \hfill
\hfill

\vspace{0.3cm}
\hfill
      \begin{subfigure}[b]{0.09\textwidth}
         \centering
         \includegraphics[width=\textwidth]{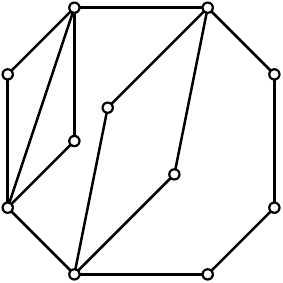}
     \end{subfigure}
 \hfill
      \begin{subfigure}[b]{0.09\textwidth}
         \centering
         \includegraphics[width=\textwidth]{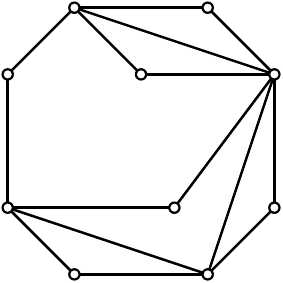}
     \end{subfigure}
      \hfill
      \begin{subfigure}[b]{0.09\textwidth}
         \centering
         \includegraphics[width=\textwidth]{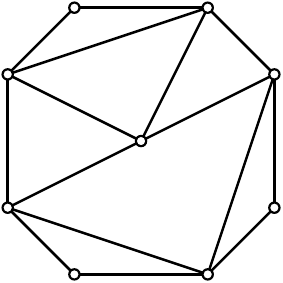}
     \end{subfigure}
 \hfill
      \begin{subfigure}[b]{0.09\textwidth}
         \centering
         \includegraphics[width=\textwidth]{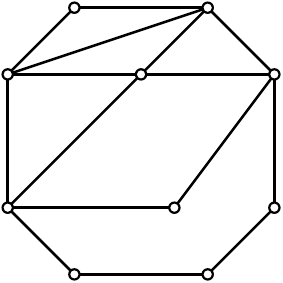}
     \end{subfigure}
 \hfill
      \begin{subfigure}[b]{0.09\textwidth}
         \centering
         \includegraphics[width=\textwidth]{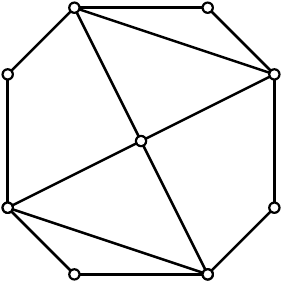}
     \end{subfigure}
 \hfill
      \begin{subfigure}[b]{0.09\textwidth}
         \centering
         \includegraphics[width=\textwidth]{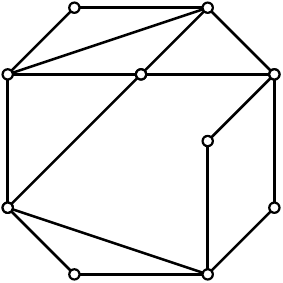}
     \end{subfigure}
      \hfill
      \begin{subfigure}[b]{0.09\textwidth}
         \centering
         \includegraphics[width=\textwidth]{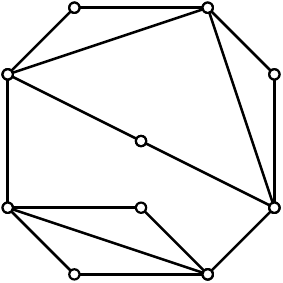}
     \end{subfigure}
     \hfill
      \begin{subfigure}[b]{0.09\textwidth}
         \centering
         \includegraphics[width=\textwidth]{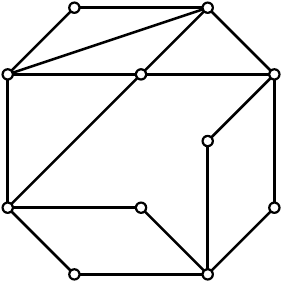}
     \end{subfigure}
     \hfill
      \begin{subfigure}[b]{0.09\textwidth}
         \centering
         \includegraphics[width=\textwidth]{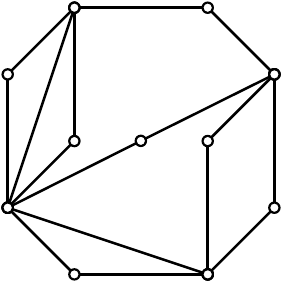}
     \end{subfigure}
     \hfill
     \hfill
     
\vspace{0.3cm}

\hfill
      \begin{subfigure}[b]{0.09\textwidth}
         \centering
         \includegraphics[width=\textwidth]{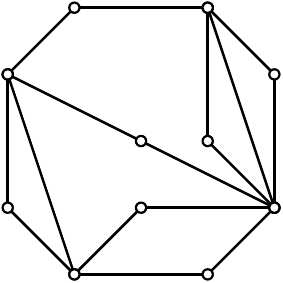}
     \end{subfigure}
     \hfill
      \begin{subfigure}[b]{0.09\textwidth}
         \centering
         \includegraphics[width=\textwidth]{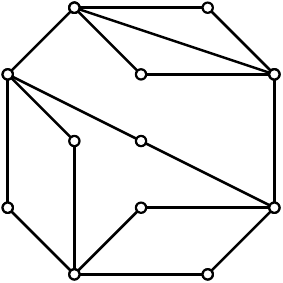}
     \end{subfigure}
     \hfill
      \begin{subfigure}[b]{0.09\textwidth}
         \centering
         \includegraphics[width=\textwidth]{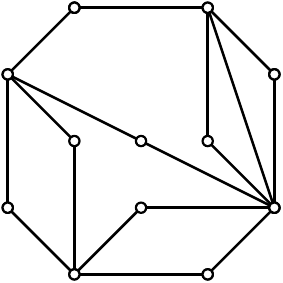}
     \end{subfigure}
      \hfill
      \begin{subfigure}[b]{0.09\textwidth}
         \centering
         \includegraphics[width=\textwidth]{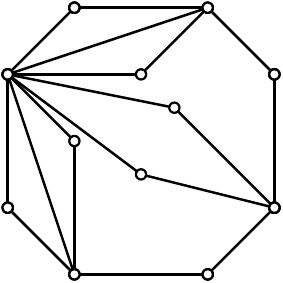}
     \end{subfigure}
      \hfill
      \begin{subfigure}[b]{0.09\textwidth}
         \centering
         \includegraphics[width=\textwidth]{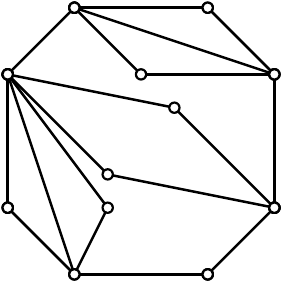}
     \end{subfigure}
     \hfill
      \begin{subfigure}[b]{0.09\textwidth}
         \centering
         \includegraphics[width=\textwidth]{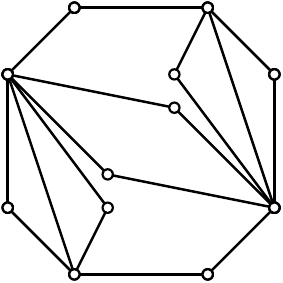}
     \end{subfigure}
     \hfill
      \begin{subfigure}[b]{0.09\textwidth}
         \centering
         \includegraphics[width=\textwidth]{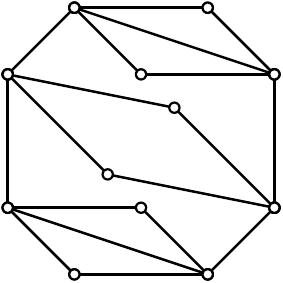}
     \end{subfigure}
      \hfill
      \begin{subfigure}[b]{0.09\textwidth}
         \centering
         \includegraphics[width=\textwidth]{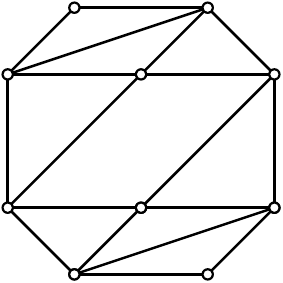}
     \end{subfigure}
      \hfill
      \begin{subfigure}[b]{0.09\textwidth}
         \centering
         \includegraphics[width=\textwidth]{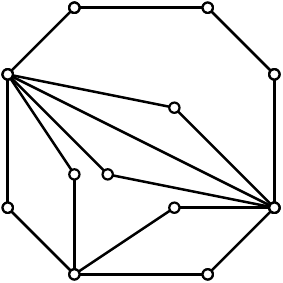}
     \end{subfigure}
      \hfill
      \hfill

      \vspace{0.3cm}
      \hfill
      \begin{subfigure}[b]{0.09\textwidth}
         \centering
         \includegraphics[width=\textwidth]{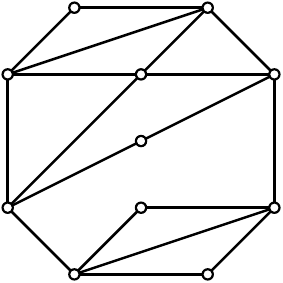}
     \end{subfigure}
     \hfill
      \begin{subfigure}[b]{0.09\textwidth}
         \centering
         \includegraphics[width=\textwidth]{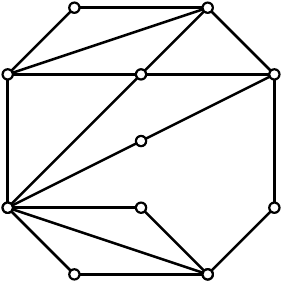}
     \end{subfigure}
     \hfill  
      \begin{subfigure}[b]{0.09\textwidth}
         \centering
         \includegraphics[width=\textwidth]{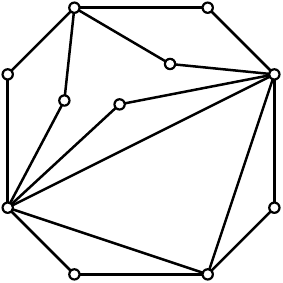}
     \end{subfigure}
      \hfill
      \begin{subfigure}[b]{0.09\textwidth}
         \centering
         \includegraphics[width=\textwidth]{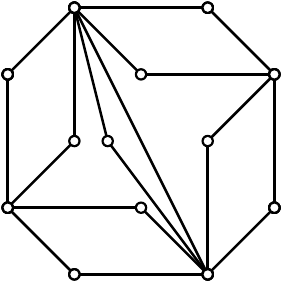}
     \end{subfigure}
     \hfill
      \begin{subfigure}[b]{0.09\textwidth}
         \centering
         \includegraphics[width=\textwidth]{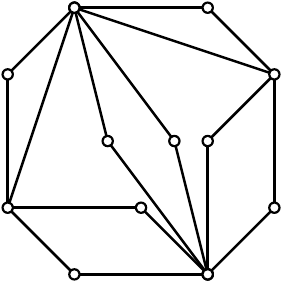}
     \end{subfigure}
     \hfill
      \begin{subfigure}[b]{0.09\textwidth}
         \centering
         \includegraphics[width=\textwidth]{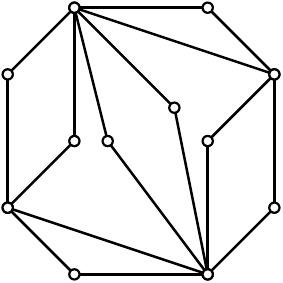}
     \end{subfigure}
     \hfill
      \begin{subfigure}[b]{0.09\textwidth}
         \centering
         \includegraphics[width=\textwidth]{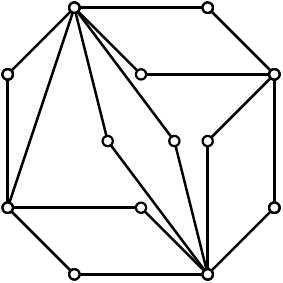}
     \end{subfigure}
     \hfill
      \begin{subfigure}[b]{0.09\textwidth}
         \centering
         \includegraphics[width=\textwidth]{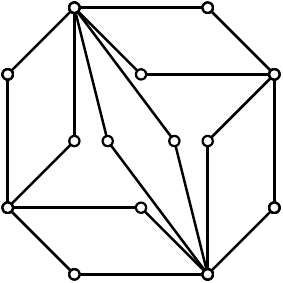}
     \end{subfigure}
     \hfill
      \begin{subfigure}[b]{0.09\textwidth}
         \centering
         \includegraphics[width=\textwidth]{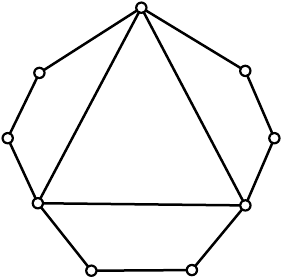}
     \end{subfigure}
      \hfill
     \hfill

    \caption{2-connected unstable duals corresponding to an embedding of genus 3 in a cubic C5EC planar graph}
    \label{fig:g3duals}
\end{figure}

\begin{question}
    Characterize the unstable duals of embeddings of genus $k\geq 2$ in a cubic planar graph, or of genus $k\geq 3$ in a cubic C5EC planar graph. 
\end{question}

Moreover, in Section \ref{section:g2}, we characterize only the unstable duals with genus 2 of a cubic C5EC planar graph. Using Enami's characterization of toroidal embeddings of a cubic 3-connected planar graph and Theorems \ref{thm:cut vertex dual} and \ref{thm:2-cuts}, we may use copies of $K_{2,2,2}$, $K_{2,2m}$ and $K_{1,1,2m-1}$, for $m \in \mathbb{N}$, to construct the unstable dual of an embedding with genus 2 of a cubic 3-connected planar graph that is not C5EC. However, a full characterization of such unstable duals remains open. 

\begin{question}
    Characterize the unstable duals of embeddings of genus 2 in a cubic 3-connected planar graph that is not C5EC. 
\end{question}

In part II of this paper, we explore applications of the unstable dual to problems related to the LCGD Conjecture, including cubic counterexamples to the conjecture and a class of cubic C5EC planar graphs whose first three terms of the genus distribution satisfy the log-concavity condition. 

\bibliographystyle{abbrv}
\bibliography{references1}
\end{document}